\documentclass[12pt]{amsart}
\usepackage[utf8]{inputenc}
\usepackage{bm}
\usepackage{fontenc}
\usepackage{wasysym}
\usepackage{amsfonts}
\usepackage{amssymb}
\usepackage{amsmath}
\usepackage{amsthm}
\usepackage{mathtools}
\usepackage{enumerate}
\usepackage{mathrsfs}
\usepackage{tikz}
\usetikzlibrary{calc}
\usepackage{marginnote}
\usepackage{xcolor,enumitem}
\usepackage{soul}
\usepackage{hyperref}
\usepackage[all,cmtip]{xy}

\newcommand{\mathscripty}{\mathscr}

\newcommand{\rs}{\mathord{\upharpoonright}}

\newcommand{\cL}{\mathcal{L}}

\newcommand{\cX}{\mathcal{X}}

\newtheorem*{rigprob*}{Rigidity Problem for uniform Roe Algebras}
\newtheorem*{rigprobcorona*}{Rigidity Problem for uniform Roe Coronas}

\newcommand{\SP}{\mathscripty{P}}

\newcommand{\cstar}{$\mathrm{C}^*$}

\newcommand{\bbL}{\mathbb L}
\newcommand{\bbP}{\mathbb P}

\newcommand{\cU}{\mathcal{U}}
\newcommand{\cZ}{\mathcal{Z}}
\newcommand{\cF}{\mathcal{F}}
\newcommand{\cP}{\mathcal{P}}
\newcommand{\bbN}{\mathbb{N}}
\newcommand{\bbZ}{\mathbb{Z}}

\newcommand{\cI}{\mathcal I}

\newcommand{\ZFC}{\mathrm{ZFC}}

\newcommand{\CH}{\mathrm{CH}}

\renewcommand{\P}{\SP}

\newtheorem{thm}{Theorem}
\newtheorem{coro}[thm]{Corollary}
\newtheorem{theorem}{Theorem}[section]
\newtheorem*{theorem*}{Theorem}
\newtheorem{proposition}[theorem]{Proposition}

\newtheorem*{proposition*}{Proposition}
\newtheorem{lemma}[theorem]{Lemma}
\newtheorem*{lemma*}{Lemma}
\newtheorem{corollary}[theorem]{Corollary}
\newtheorem*{corollary*}{Corollar}

\newtheorem*{fact*}{Fact}

\theoremstyle{definition}
\newtheorem{definition}[theorem]{Definition}
\newtheorem*{definition*}{Definition}
\newtheorem{claim}[theorem]{Claim}
\newtheorem*{claim*}{Claim}

\newtheorem*{conjecture*}{Conjecture}

\newtheorem{question}[theorem]{Question}

\theoremstyle{remark}

\newtheorem{examples}[theorem]{Examples}
\newtheorem*{example*}{Example}
\newtheorem{remark}[theorem]{Remark}
\newtheorem*{remark*}{Remark}

\newtheorem*{note*}{Note}
\newtheorem*{question*}{Question}

\newtheorem*{acknowledgements*}{Acknowledgements}

\DeclareMathOperator{\Fin}{Fin}
\DeclareMathOperator{\OFin}{\emptyset\otimes\Fin}

\DeclareMathOperator{\Th}{Th}
 
\newcommand{\calL}{\mathcal L}

\newcommand{\twon}{\{0,1\}^n}
\newcommand{\twoN}{\{0,1\}^\bbN}
\newcommand{\fc}{\mathfrak c}

\newcommand{\bt}{\mathbf t}

\newcommand{\leqphi}{\leq_\varphi}
\newcommand{\leqRB}{\leq_{\textrm{RB}}}

\usepackage{enumitem}

\begin{document}

\title{Saturation of reduced products}

\author[BDB]{Ben De Bondt}
\address[BDB]{
Institut de Math\'ematiques de Jussieu - Paris Rive Gauche (IMJ-PRG)\\
Universit\'e Paris Cit\'e\\
B\^atiment Sophie Germain\\
8 Place Aur\'elie Nemours \\ 75013 Paris, France}
\email{ben.de-bondt@imj-prg.fr}
\urladdr{https://perso.imj-prg.fr/ben-debondt/}

\author[IF]{Ilijas Farah}
\address[IF]{Department of Mathematics and Statistics\\
York University\\
4700 Keele Street\\
North York, Ontario\\ Canada, M3J 1P3\\
and 
Ma\-te\-ma\-ti\-\v cki Institut SANU\\
Kneza Mihaila 36\\
11\,000 Beograd, p.p. 367\\
Serbia}
\email{ifarah@yorku.ca}
\urladdr{https://ifarah.mathstats.yorku.ca}
\thanks{I.F. is partially supported by NSERC}

\author[AV]{Alessandro Vignati}
\address[AV]{
Institut de Math\'ematiques de Jussieu - Paris Rive Gauche (IMJ-PRG)\\
Universit\'e Paris Cit\'e\\
B\^atiment Sophie Germain\\
8 Place Aur\'elie Nemours \\ 75013 Paris, France}
\email{ale.vignati@gmail.com}
\urladdr{http://www.automorph.net/avignati}

\dedicatory{Dedicated to \v Zarko Mijajlovi\' c on the occasion of his 75th birthday.}
\date{\today}
 
\maketitle
\begin{abstract}
We study reduced products $M=\prod_n M_n/\Fin$ of countable structures in a countable language associated with the Fr\'echet ideal. We prove that such $M$ is $2^{\aleph_0}$-saturated if its theory is stable and not $\aleph_2$-saturated otherwise (regardless of whether the Continuum Hypothesis holds). 
This implies that $M$ is isomorphic to an ultrapower (associated with an ultrafilter on $\bbN$) if its theory is stable, even if the CH fails. We also improve a result of Farah and Shelah and prove that there is a forcing extension in which such reduced product $M$ is isomorphic to an ultrapower if and only if the theory of $M$ is stable. All of these conclusions apply for reduced products associated with $F_\sigma$ ideals or more general layered ideals.	
We also prove that a reduced product associated with the asymptotic density zero ideal $\cZ_0$, or any other analytic P-ideal that is not $F_\sigma$, is not even $\aleph_1$-saturated if its theory is unstable. 
\end{abstract}

Saturated\footnote{Throughout, the term saturated is taken in its model-theoretic meaning, not to be confused with saturated ideals used in set theory. See also \S\ref{S.terminology.saturation} for additional clarification.} models hold a prominent position in the realm of model theory. They can be constructed through transfinite recursion, where meticulous bookkeeping ensures the realization of all pertinent types in the ultimate model. Alternatively, a saturated model may be constructed as an ultrapower. In both scenarios, these models are colossal entities, constructed using the Axiom of Choice and aptly referred to as `monster models', for a compelling reason. Notably, only limited segments of these expansive structures are typically showcased in any specific proof.
It is not even necessary to know that a saturated model of a theory at hand exists in order to apply an argument using saturated models; see \cite{halevi2023saturated}. Indeed an unstable theory $T$ may not have saturated models at all (see \S\ref{S.existence}). 

Our interest in saturated models has completely different origin. The impetus for this work comes from study of automorphism groups of definable models (primarily reduced products associated with $\Fin$ and other analytic ideals on $\bbN$). More precisely, the motivation for the work presented here stems from the following question: \emph{what makes some reduced products rigid and others saturated?} It is well-known (\cite{JonssonOlin}) that every reduced product $\prod_n M_n/\Fin$, where $\Fin$ is the Fr\'echet ideal (the ideal of finite subsets of $\bbN$), is $\aleph_1$-saturated (i.e., countably saturated, meaning every countable consistent type is realised; this is different from countably saturated as defined in \cite[p. 100]{ChaKe}).\footnote{See \S\ref{S.terminology.saturation}.}
Therefore a standard back-and-forth argument due to Keisler shows that if the Continuum Hypothesis ($\CH$) holds, 
then for all $M_n$ of cardinality not greater than\footnote{By $\fc$ we denote the cardinality of the continuum, $2^{\aleph_0}$.} $\fc$:
\begin{enumerate}[label=(\alph*)]
\item\label{itm:a}the reduced product $\prod_n M_n/\Fin$ has $2^{\fc}$ automorphisms,
\item\label{itm:b} $\prod_n M_n/\Fin$ is isomorphic to the ultrapower of any of its elementary submodels associated with a nonprincipal ultrafilter on~$\bbN$,
\item\label{itm:c} if $\prod_n N_n/\Fin$ is another reduced product with all $N_n$ of cardinality not greater than $\fc$, then $\prod_n M_n/\Fin$ and $\prod_n N_n/\Fin$ are isomorphic if and only if they are elementarily equivalent. 
\end{enumerate}
On the other hand, if $\CH$ fails then different reduced products $\prod_n M_n/\Fin$ of small structures (even restricting just to countable structures $M_n$) exhibit different behaviour, which for some reduced products is markedly different than \ref{itm:a},\ref{itm:b},\ref{itm:c} above. The prime (and historically first) example of this phenomenon is a seminal result of Shelah, that there is a forcing extension in which the reduced power of the two-element Boolean algebra $ \{0,1\}^\bbN/\Fin$ has only trivial automorphisms~(\cite{Sh:Proper}). In contrast, the group $ (\bbZ/2\bbZ)^\bbN/\Fin$ is saturated and has $2^{\fc}$ automorphisms, provably in $\ZFC$. The underlying structure in both cases is the same, $\cP(\bbN)/\Fin$, but their languages are different.

For structures of cardinality not greater than $\fc$, our main results identify model-theoretic stability as \emph{the} dividing line between: on one side at most $\aleph_1$-saturation, and on the other side the maximal level of saturation possible (independently of whether $\CH$ holds). Stability is a central property of first-order theories that was isolated by Shelah in \cite{shelah1969}, see also \cite{shelah2009classification}, from the proof of Morley's Categoricity Theorem (\cite{Morley.Cat}) and serves as a fundamental dividing line for first order theories. 
Stable theories are tame, while instability (in its various guises and degrees) allows for coding combinatorial structures of varying complexities. This paradigm extends smoothly to logic of metric structures (\cite{BYU:ContStab}, \cite{FaHaSh:Model1} and \cite{FaHaSh:Model2}). 
The following is our first main theorem, and it generalizes the behaviour of the group $ (\bbZ/2\bbZ)^\bbN/\Fin$ mentioned earlier. 

\begin{thm}\label{T.saturated}
If $M_n$, for $n\in \bbN$, are structures in a countable language and the theory of $M=\prod_n M_n/\Fin$ is stable, then $M$ is $\fc$-saturated. In particular, if $|M_n|\leq \fc$ for all $n$, then $M$ is saturated. 
\end{thm}

In Theorem~\ref{T.cs} we prove analogous results for a wider family of ideals, including all $F_\sigma$ ideals\footnote{Here $\cP(\bbN)$ is endowed with its natural Cantor space topology.} that include $\Fin$ (see however Question~\ref{C.atomless}). The following shows that stability is the correct dividing line in this context (at least in case of $\Fin$). 

\begin{thm}\label{T.non-saturated}
 If $M_n$, for $n\in \bbN$, are structures in a countable language and the theory of $M=\prod_n M_n/\Fin$ is unstable, then $M$ is not $\aleph_2$-saturated. In particular, if $\CH$ fails then $M$ is not saturated. 
\end{thm}

This is a special case of Corollary~\ref{T.non-saturated+}, in which $\Fin$ is replaced by an arbitrary analytic ideal that extends $\Fin$. The presence of $\aleph_2$ is explained by the existence of a Hausdorff gap in $\cP(\bbN)/\Fin$, and our proof gives a finer `non-saturation transfer theorem’ (Theorem~\ref{T.non-saturated++}). This implies for example that if a reduced product associated with $\cZ_0$ (the ideal of asymptotic density zero) has a theory which is unstable, then it is not even countably saturated (Corollary~\ref{C.P-ideal}). 

Moving on to the situation in which CH fails and strong rigidity results hold, the failure of saturation is a necessary condition for the analog of Shelah’s result about triviality of automorphisms of $\cP(\bbN)/\Fin$, but it is not sufficient. For example, in Cohen’s original model in which $\CH$ fails $\cP(\bbN)/\Fin$ is not saturated but it has $2^{\fc}$ nontrivial automorphisms (\cite{ShSte:Non-trivial}). Numerous rigidity results for quotients along similar lines with Shelah’s have been proven for quotient Boolean algebras, \v Cech--Stone remainders, and \cstar-algebras; see \cite{farah2022corona} for a survey. More recently, in a companion paper (\cite{de2023trivial}) we showed that under forcing axioms a version of Shelah's result on trivial automorphisms holds for reduced products of countable structures in other categories: fields, linear orderings, and sufficiently random graphs. 

It is not surprising, but worth mentioning, that our Theorem~\ref{T.saturated} and Theorem~\ref{T.non-saturated} together imply that the full saturation of reduced products is a property of the theory (unless trivially implied by $\CH$), and is therefore absolute between models of $\ZFC$ and the negation of $\CH$ (this is not automatic because, unlike the original structures, the reduced products themselves change whenever a new real is added to the universe). 
We already mentioned the trivializing effect of $\CH$: it implies that both ultraproducts and reduced products of countable structures associated with $\Fin$ are saturated, leading to the phenomena in \ref{itm:a},\ref{itm:b},\ref{itm:c} above. In particular, $\CH$ implies that, if all $M_n$ have cardinality not greater than~$\fc$, then $\prod_n M_n/\Fin$ is isomorphic to an ultrapower of any of its elementary submodels (item \ref{itm:b}). We also show that for this property stability separates structures `with $\CH$-like behaviour' from others.
By \cite[Theorem~5.6]{FaHaSh:Model2}, an ultrapower of a countable structure associated with a nonprincipal ultrafilter on $\bbN$ is saturated, provided that its theory is stable. Thus Theorem~\ref{T.saturated} implies the following $\CH$-like behaviour of reduced products with stable theory. 

\begin{coro}\label{C.stable} Suppose that $M_n$, for $n\in \bbN$, are structures of cardinality not greater than $\fc$ in a countable language. If the theory of $M=\prod_n M_n/\Fin$ is stable, then $M$ is isomorphic to the ultrapower associated with any nonprincipal ultrafilter on $\bbN$ of any of its countable elementary submodels, provably in $\ZFC$. \qed 
\end{coro}

The conclusion of this corollary, as well as that of Theorem~\ref{T.saturated}, are of course trivial for uncountably categorical  theories (since all uncountable models of such theory are, by Morley's theorem, saturated, see e.g., \cite[\S 7]{ChaKe}). 

Massive models such as reduced products and ultraproducts are in some arguments interchangeable. This is especially useful in the theory of \cstar-algebras; see the discussion in the introduction to \cite{farah2022betweenI}, which contains more precise results on the relation between ultrapowers and reduced powers in the presence of CH. Motivated by the need to explain this phenomenon, in \cite[Theorem~C(1)]{farah2022betweenII} it was proven that if theories of structures~$M_n$ have the so-called robust order property (a strengthening of the order property that passes to reduced products) then in some forcing extension, $\prod_{n} M_n/\Fin$ is not isomorphic to any ultraproduct $\prod_n N_n/\cU$, where $\cU$ is a nonprincipal ultrafilter on $\bbN$. We refine this result and prove that the stability again provides the precise dividing line (Theorem~\ref{T.nonisomorphic}). 

\subsection*{Acknowledgements}
This work was completed during a visit of IF to the Institut de Math\'ematiques de Jussieu, in Paris. This visit was funded by Universit\'e Paris Cit\'e and IF’s NSERC grant, and the authors would like to thank both institutions. We are indebted to Sylvy Anscombe, Tom\' as Ibarlucia, and Boban Veli\v{c}kovi\'c for conversations. 

\section{Preliminaries}\label{S.preliminaries}
Not much happens in the current section, and the reader may want to skim it and refer to it later on as needed. 

\subsection{Notational conventions}\label{S.Notational}
 We list some of the conventions used throughout. By $\bar a$ we denote a tuple of an unspecified, but syntactically appropriate, size. (In other words, if $\varphi(\bar x,\bar y)$ is a formula and we write $\varphi(\bar a, \bar b)$, it is understood that $\bar a$ is of the same sort as $\bar x$ and $\bar b$ is of the same sort as $\bar y$.) 
If $\varphi(\bar x,\bar y)$ is a formula and $M$ a structure, we write ($\bar{a} $ and $\bar{b}$ are assumed to be of the appropriate sort) 
\[
\varphi(M,\bar b)=\{\bar{a}\mid M \vDash \varphi(\bar a,\bar b)\}.
\]

\subsubsection{Language}\label{S.terminology.saturation}
We fix a countable language, understood to be the language of all formulas and structures considered. Hence all structures and formulas mentioned in the same context are assumed to share the same language, and explicitly we refer to the language only when necessary (as e.g., in~\ref{S.Types}). 

\subsubsection{Types, saturation, compactness}\label{S.Types}
The types that we consider are not necessarily complete (as in e.g., \cite[Definition~4.4.1]{Mark:Model}; while many authors assume that every type is complete, this convention would make formulations of some of our results look awkward). 	With this convention, if $X$ is a subset of a structure $M$, then every type over $X$ is also a type over $M$. This differs from conventions used in stability theory (\cite{Pillay.IntroStab}) but will considerably simplify the presentation of our results and proofs.

If $M$ is a structure in language $\calL$ and $X\subseteq M$ then an \emph{$n$-type over $X$} is a consistent set of formulas in language $\calL(X)$, obtained from $\calL$ by adding names for elements of $X$ whose free variables are included in $x_j$, for $j<n$. If $\kappa$ is an infinite cardinal then a structure $M$ is \emph{$\kappa$-saturated} if for every $X\subseteq M$ with $|X|<\kappa$, every consistent type over $X$ is realised in $M$. It is \emph{saturated} if it is $|M|$-saturated. 

If $M$ is a structure and every (consistent) type over $M$ of cardinality strictly smaller than $\kappa$ is realised in $M$ then $M$ is said to be \emph{$\kappa$-compact}. It is \emph{compact} if it is $|M|$-compact. In general, $\kappa$-compactness is weaker than $\kappa$-saturation. However, if the language is countable (which is our standing assumption), every type over an infinite $X\subseteq M$ has cardinality $|X|$, hence there is no difference between the two notions. Our results will be stated in terms of (the more common) $\kappa$-saturation, but if the assumption on the countability of the language is dropped, all results remain valid if `$\kappa$-saturated' is replaced with $\kappa$-compact' everywhere. 

We would normally much prefer using the term `countable saturation' to `$\aleph_1$-saturation'. This terminology is great if all that one has to worry about are countable saturation and full saturation. 
However, in the present paper we have to explicitly refer to $\aleph_2$-saturation, making the extension of the preferred terminology unfeasible.

\subsubsection{Arity}\label{S.arity}
As is well known, in the definition of ($\kappa$-)saturation it suffices to consider only 1-types. The same applies to ($\kappa$-)compactness. In general, considering only tuples of arity 1 does not lead to loss of generality, by passing to the structure $M^{\textrm{eq}}$ of the imaginaries over $M$; see e.g., \cite[\S III.6]{shelah1990classification}.

\subsubsection{Ideals, reduced products}
If $\cI$ is an ideal on a set $I$, then the coideal of $\cI$-positive sets is
\[
\cI^+=\{A\subseteq I\mid A\notin \cI\}.
\]
By $A\Subset B$ we denote the assertion that $A$ is a finite subset of $B$.

\begin{definition}\label{def:redprod}
	If $\cI$ is an ideal on a set $I$ and $M_i$, for $i\in I$ are structures of the same language, the reduced product $\prod_i M_i/\cI$ associated with $\cI$ is defined as follows. Its universe is the quotient of $\prod_i M_i$ by the equivalence relation $\sim_{\mathcal I}$ defined as
\[
(a_i)\sim_{\mathcal I}(b_i)\iff \{i\in I\mid a_i\neq b_i\}\in\mathcal I,
\] 
for all $(a_i),(b_i)\in \prod_i M_i$. 

The following convention is so convenient (pun not intended) that we introduce it in the middle of the definition of a reduced product. If $\bar a$ is a $k$-tuple in a reduced product $\prod_i M_i/\cI$, then we can choose a representing sequence $\tilde{a} \in \prod_i (M_i)^k$ for $\bar a$, with $i$-th element denoted $\tilde{a}_i \in (M_i)^k$, for $i\in I$.\footnote{Since we are working under the assumption that the Axiom of Choice holds, arbitrarily many such choices can be made simultaneously.}

Let $M=\prod_iM_i/\cI$. If $f$ is a function symbol, then its interpretation is defined by ($\bar a$ is of the appropriate sort) 
\[
f^M(\bar a)=b\iff \{i \mid f^{M_i}(\tilde a_i)\neq \tilde{b}_i\}\in \cI.
\]
If $R$ is a relation symbol, then its interpretation is defined by ($\bar a$ is of the appropriate sort)
\[
R^M(\bar a)\iff \{i \mid \lnot R^{M_i}(\tilde a_i)\}\in \cI.
\]
With these interpretations, $M$ is now a structure in the same countable language of the $M_i$'s.

In case when $N=M_i$ for all $i\in I$, we refer to the $\cI$-reduced power, denoted by $N^I/\cI$ or sometimes $N^\cI$. 
\end{definition}

Definition~\ref{def:redprod} is interchangeable with the definition of a reduced product with respect to a filter (see e.g., \cite[\S 4.1.6]{ChaKe}); more precisely, if $\cF=\cI^*$, the filter dual to $\cI$, then $\prod_i M_i/\cI$ and $\prod_i M_i/\cF$ are naturally isomorphic. If $\mathcal I$ is dual to an ultrafilter $\cU$, then $\prod_i M_i/\cI$ is the ultraproduct $\prod_i M_i/\cU$.

\subsection{Gaps} \label{S.Gaps} 
Theorem~\ref{T.non-saturated} is proven by injecting the reason why $\cP(\bbN)/\Fin$ is not $\aleph_2$-saturated into the given reduced product. This reason is provided by a well-known 1908 result of Hausdorff that we nevertheless proceed to describe, but in a greater generality. 

Even if a binary relation $\leqphi$ on some structure $M$ is not a partial ordering, one can define gaps and limits in $(M,\leqphi)$ analogously to the standard definition. To wit:

\begin{definition} \label{Def.Gaps} 
	%
	Suppose that $(M,\leqphi)$ is a set equipped with a binary relation $\leqphi$,\footnote{It is not assumed that $\leqphi$ is transitive or antisymmetric, and the choice of notation is suggestive of the connection with \S\ref{S.Order}.} and let $\kappa$ and $\lambda$ be infinite cardinals. A collection $(a_\alpha,b_\beta)_{\alpha< \kappa,\beta<\lambda}$ of elements of $M$ is said to be a $(\kappa,\lambda)$-gap in $(M,\leqphi)$ if for all $\alpha<\beta<\kappa$ and $\gamma<\delta<\lambda$ we have that
	\[
	a_\alpha\leqphi a_\beta\leqphi b_\delta \leqphi b_\gamma,
	\]
	and there is no $c$ such that $a_\alpha \leqphi c \leqphi b_\beta$ for all $\alpha<\kappa$ and $\beta<\lambda$.

	A collection $(a_\alpha,b)_{\alpha< \kappa}$ of elements of $M$ is said to be a \emph{$\kappa$-limit} if for all $\alpha<\beta<\kappa$ we have that 
	\[
	a_\alpha\leqphi a_\beta \leqphi b
	\]
	and if $c$ is such that $a_\alpha\leqphi c\leqphi b$ for all $\alpha$ then $b\leqphi c$. 
	Naturally, a collection $(a,b_\beta)_{\beta< \kappa}$ of elements of $M$ is also said to be a (downwards directed) \emph{$\kappa$-limit} 
	if for all $\alpha<\beta<\kappa$ we have that 
	\[
	a\leqphi b_\beta \leqphi b_\alpha 
	\]
	and if $c$ is such that $a\leqphi c\leqphi b_\alpha$ for all $\alpha$ then $c\leqphi a$. 
\end{definition}

If $A$ and $B$ are subsets of a set $I$ and $\mathcal I\subseteq\mathcal P(I)$ is an ideal, we write $A\subseteq^\mathcal IB$ if $A\setminus B\in\mathcal I$. (When $\mathcal I=\Fin$, this relation is commonly denoted by~$\subseteq^*$.) 
The special case of Definition~\ref{Def.Gaps} when $M=\cP(\bbN)$ and $\leqphi$ is $\leq^*$ or $\leq^{\cI}$ for some $\cI$ is of particular importance. 
Some authors call $\kappa$-limits $(\kappa,1)$ gaps, although this is a misnomer since such `gap' is obviously split by the unique element of its singleton side. 

By the result of Hausdorff just referred to there exists an $( \aleph_1, \aleph_1)$-gap in $\cP(\bbN)/\Fin$ (see e.g., \cite[Theorem~9.3.7]{Fa:STCstar} for a proof), and it was proven in \cite[Corollary~14]{To:Gaps} that there exists an $(\aleph_1,\aleph_1)$-gap in $\cP(\bbN)/\cI$ for every analytic ideal $\cI$ that includes $\Fin$.\footnote{Some authors distinguish between the cardinal $\aleph_\alpha$ and the ordinal $\omega_\alpha$, but there is no compelling reason to do so (see the discussion in \cite[\S A.4]{Fa:STCstar}).}

\begin{remark}
We have defined a gap as a pair where one `side’ of the gap is increasing and the other is decreasing. While this picture is more convenient for our present purposes, some authors (including some of the authors of the present paper) generally prefer to think of gaps as pairs with two increasing and orthogonal `sides’ that cannot be separated, without even requiring them to be linearly ordered by $\subseteq^\cI$ (or $\leqphi$ for that matter); see \cite[Notes for Chapter 9]{Fa:STCstar} for a discussion. 
\end{remark}

\section{Simple Cover Property of Reduced Products} \label{S.SimpleCoverProperty}
This preparatory section consists of rehashing results of Palyutin as presented by Palmgren that are used to analyze stability in reduced products. \L o\'s's theorem is the fundamental tool for computing the theory of an ultraproduct over an ultrafilter $\mathcal U$. It states that a sentence is true in an ultraproduct if and only if it is true $\mathcal U$-often on fibers. For general reduced products, this is no longer true. For example, a reduced product of fields $\prod_i K_i/\mathcal I$ is a field only when $\mathcal I$ is a maximal ideal. In a series of papers (\cite{palyutin1,palyutin2}), Palyutin isolated the notion of an $h$-formula. These formulas satisfy the exact analog of \L o\'s's theorem for arbitrary reduced products. We use Palyutin's work, but our presentation, terminology, and notation in this and the next subsection are based on Palmgren's~\cite{palmgren}.

\begin{definition}[$h$-formulas] \label{Def.h-formulas} 
The class of \emph{$h$-formulas} is the smallest class of formulas $\mathcal C$ containing all atomic formulas and such that if $\varphi$ and $\psi$ belong to $\mathcal C$, then so do
\[
\varphi\wedge\psi,\,\,\exists x\varphi,\,\,\forall x\varphi\text{ and }\exists x\varphi\wedge\forall x(\varphi\rightarrow\psi).
\]
\end{definition}

$h$-formulas should not be confused with Horn formulas (since we will not need Horn formulas, we will not define them). While by a result of Galvin, a formula is equivalent to a Horn formula if and only if it is preserved by all reduced products (\cite{galvin1970horn}, \cite[\S 6.3]{ChaKe}), in case of $h$-formulas we have the following (much easier) result of Palyutin. 

\begin{theorem}[\cite{palyutin2}]\label{T.Palyutin}
If $\cI$ is an ideal on a set $I$, $M_i$, for $i\in I$, are structures in a countable language, $M=\prod_i M_i/\cI$, and $\varphi(\bar x)$ is an $h$-formula, then
\begin{align*}
M&\models \varphi(\pi(\bar a))\text{ if and only if } \{i\in I \mid M_i\models \neg \varphi(\bar a_i)\}\in \cI,\\ M&\models \neg\varphi(\pi(\bar a))\text{ if and only if } \{i\in I \mid M_i\models \neg \varphi(\bar a_i)\}\notin \cI, 
\end{align*}
where $\pi\colon\prod_iM_i\to M$ denotes the quotient map. 
\end{theorem}

\begin{proof}
By induction on the complexity of $\varphi$. 
\end{proof}

\subsection{Simple cover property} The following first-order
property, isolated by Palmgren (\cite{palmgren}), characterizes theories of reduced products associated with $\Fin$ (Corollary~\ref{C.Preservation} below). 

\begin{definition} \label{D.simple} 
A structure $M$ has the \emph{simple cover property} if for every $m\geq 1$ and all $h$-formulas $\varphi(\bar x,\bar y)$ and $\psi_j(\bar x,\bar y)$, for $j<m$, the following distributivity law holds whenever $\bar a$ is a tuple in $M$ of the appropriate sort:
\[
M\models (\forall \bar x)\left(\varphi(\bar x,\bar a) \rightarrow \bigvee_{j<m}\psi_j(\bar x,\bar a) \right)\rightarrow \bigvee_{j<m}(\forall \bar x)(\varphi(\bar x,\bar a) \rightarrow \psi_j(\bar x,\bar a)).
\]
A theory $T$ has the simple cover property if all of its models have it.
\end{definition}
In \cite[p. 380]{palyutin1}, the instance of the simple cover property as in the displayed formula in Definition~\ref{D.simple} is denoted $\alpha(\varphi,\psi_0,\dots, \psi_{m-1},\bar y)$. 
The following is \cite[Proposition~5]{palmgren} (compare with McKinsey’s Lemma, \cite[Lemma~9.1.7]{Hodg:Model}).

\begin{lemma} \label{L.scp} 
If $\cI$ is an ideal on a set $I$ and $\cP(I)/\cI$ is atomless then any reduced product $M=\prod_i M_i/\cI$ has the simple cover property. \qed 
\end{lemma}

The following is \cite[Proposition~1]{palyutin1}, see also \cite[Lemma 3]{palmgren}. 

\begin{lemma}\label{lemma:booleanh} 
If $M$ has the simple cover property, then for every formula $\varphi$ there is a Boolean combination of $h$-formulas $\varphi'$ such that $M\models \varphi\leftrightarrow \varphi'$. \qed 
\end{lemma}

We interrupt the proceedings to record a simple corollary, hitherto overlooked in the literature to the best of our knowledge, yet noticed in \cite[Theorem 2.13]{fronteau2023produits}. It is well-known that a consistent theory $T$ is preserved by homomorphisms if and only if it has a set of axioms consisting of positive sentences (\cite[Theorem~5.2.13]{ChaKe}). In the following we offer a related, but more specific result. 

\begin{corollary} \label{C.Preservation} 
For a complete theory $T$ the following are equivalent. 
\begin{enumerate}
	\item \label{1.Preservation} $T$ is preserved under reduced products.
	\item \label{2.Preservation} $T$ has the simple cover property. 
	\item \label{3.Preservation} $T$ is the theory of $\prod_i M_i/\Fin$ for some structures $M_i$, for $i\in \bbN$, of the same language. 
\end{enumerate}
\end{corollary}

\begin{proof} \eqref{1.Preservation} trivially implies \eqref{3.Preservation} and 
	 \eqref{3.Preservation} implies \eqref{2.Preservation} is Lemma~\ref{L.scp}.

	To prove that \eqref{2.Preservation} implies \eqref{1.Preservation}, assume that $T$ has the simple cover property, $M_i\models T$ for $i\in I$, and~$\cI$ is an ideal on $I$. Fix a sentence $\psi$ in the language of~$T$. We claim that $\prod_i M_i/\cI\models \psi$ if $\psi\in T$. If $\psi$ is an $h$-formula, this follows from Theorem~\ref{T.Palyutin}. Otherwise, $\psi$ is a Boolean combination of finitely many $h$-formulas by Lemma \ref{lemma:booleanh}. 
 
Since $T$ is complete, for every formula $\psi$, the set $X_\psi=\{i\mid M_i\models\lnot \psi\}$ is either empty or equal to $I$. Thus we need to prove that $\prod_i M_i/\cI\models \psi$ if and only if $X_\psi=\emptyset$ by induction on the complexity of $\psi$. This is true for $h$-formulas. Since $X_\psi\cap X_\varphi=X_{\psi\cap \varphi}$, the set of formulas for which this is true is closed under conjunctions. Finally, $X_\varphi=\emptyset$ if and only if $X_{\lnot\varphi}=I$ if and only if $X_{\lnot\varphi}\neq \emptyset$.
\end{proof} 

A semantic analogue of Corollary~\ref{C.Preservation} is given in \cite[Proposition~3.5]{farah2022betweenI}. It provides a functor $K$ from $\calL$-structures into $\calL$-structures with the property that if $\cI$ is an ideal and $\cP(\bbN)/\cI$ is an atomless Boolean algebra, then $K(A)$ is elementarily equivalent to $A^\bbN/\cI$. From the definition of $K$ it is obvious that $K^2(A)\cong K(A)$ for all~$A$. This has a useful consequence that, assuming $\CH$, for a countable (or separable) $A$, $A^\bbN/\Fin$ is isomorphic to the ultrapower of $K(A)$ associated with any nonprincipal ultrafilter on $\bbN$ (\cite[Theorem~E]{farah2022betweenI}). The latter fact has an even more useful consequence, that the quotient map from $A^\bbN/\Fin$ onto $A^\bbN/\cU$ has a right inverse if $\cU$ is a P-point ultrafilter (\cite[Theorem~B]{farah2022betweenI}).

Another ingredient in the proof of Theorem~\ref{T.cs} asserts that the theory of a reduced product is stable if and only if all $h$-formulas are stable in this theory.

\subsection{Stability and order property in models with simple cover property}\label{S.Order}
	
	In this section we analyze reduced products whose theory is unstable, towards proving Theorem~\ref{T.non-saturated} and Theorem~\ref{T.non-saturated++}. Even though originally introduced as a smallness property of the set of types over a small set, the following (equivalent) definition of stability fits our purposes (see \cite[Definition~0.7 and Theorem~2.15]{Pillay.IntroStab}).

\begin{definition}
Let $T$ be a complete first-order theory, and let $\varphi(\bar x,\bar y)$ be a formula of the same language. 
\begin{itemize}
\item Let $k\in\bbN$. $\varphi(\bar x,\bar y)$ has the $k$-order property (relative to $T$) if there exist $M$, a model of $T$, and tuples of the appropriate sort $\bar a_i$ and $\bar b_i$ in $M$, for $i<k$, such that $M\models\varphi(\bar a_i,\bar b_j)$ if and only if $i\leq j$;
\item $\varphi(\bar x,\bar y)$ has the order property (relative to $T$) if it has the $k$-order property for every $k\in\bbN$, and $\varphi$ is stable (relative to $T$) otherwise;
\item $T$ is stable if every formula of its language is stable relative to~$T$. 
\end{itemize}
\end{definition}
Note that, since $T$ is complete, the $k$-order property for a fixed $k$ does not depend on the model of $T$ we choose.

If $M$ is a structure and $k\geq 1$, a binary relation ${\leq}$ on $k$-tuples in $M$ is said to be \emph{$h$-definable} if there is a $2k$-ary $h$-formula $\varphi(\bar x,\bar y)$ such that for all $\bar x,\bar y\in M^k$ we have that $\bar x\leq\bar y$ if and only if $\varphi(\bar x,\bar y)$. 
We will write 
\[
\bar x\leqphi \bar y
\] 
if $\varphi(\bar x, \bar y)$ (suppressing the structure when there is no danger of confusion), with the understanding that $\leq_\varphi$ need not be an ordering on the $k$-tuples in the ambient model $M$. 

\begin{definition}
We say that $\leqphi$ is reflexive on its domain if $\bar a\leqphi\bar b$ implies $\bar a\leqphi\bar a$ and $\bar b\leqphi\bar b$. (Of course, replacing $\varphi(\bar x,\bar y)$ with the formula $(\bar x=\bar y)\lor \varphi(\bar x,\bar y)$ would provide a relation reflexive on its domain, but we have to be careful because the latter is not obviously an $h$-formula.) By Theorem~\ref{T.Palyutin}, transitivity of an $h$-definable relation survives a reduced product construction. 

We call a set $C$ of $k$-tuples in the structure $M$ a $\leqphi$-chain if the restriction of $\leqphi$ to $C$ is a linear order on this set. If $\leqphi$ is a transitive relation which is reflexive on its domain and is definable in a complete theory $T$, then its \emph{height} is the supremum of lengths of finite $\leqphi$-chains in some model of $T$. 
\end{definition}

Since `there exists an $\leqphi$-chain of length $n$’ is a first-order statement, the height of $\leqphi$ in $T$ does not depend on the choice of a model. 

\begin{lemma} \label{L.dichotomy} 
Let $T$ be a complete theory with the simple cover property, and let $\leqphi$ be an $h$-definable transitive relation which is reflexive on its domain. Then the height of $\leqphi$ is either 0, 1 or infinite.
\end{lemma} 

\begin{proof} 
For convenience we assume that $\varphi$ is a binary formula, i.e., $k=1$ (in case we deal with tuples, the proof is only notationally different). By Corollary~\ref{C.Preservation}, $T$ is preserved by taking reduced powers, and it suffices to prove that if there are a model $M$ of $T$ and a $\leqphi$-chain of length $2$ in $M$, then in $M^\bbN/\Fin$ there is an infinite $\leqphi$-chain, as the latter is a model of $T$. We can therefore assume that there exist elements $a,b$ in $M$ such that $a\leqphi b$ and not $b\leqphi a$ and moreover $\leqphi$ is reflexive on $\{a,b\}$.

Let $\bbN=\bigsqcup_n A_n$ be a partition of $\bbN$ into infinite sets. Define $c(n)$ by its representing sequence, where 
\[
c(n)_i=\begin{cases}\tilde b_i&\text{ if } i\in\bigcup_{j\leq n}A_j\\
\tilde a_i&\text{ if } i\in\bigcup_{j> n}A_j
\end{cases}.
\]
Since $h$-formulas pass to reduced products by Theorem~\ref{T.Palyutin}, we have that $c(n)$, for $n\in \bbN$, is a $\leqphi$-chain.
\end{proof}

In Theorem~\ref{thm:palyutinstab} below we prove that the theory of a model with the simple cover property can fail to be stable only in a very specific way. It is well-known that every complete unstable theory $T$ either has the strong order property (SOP) or the independence property (IP). A theory $T$ is said to have the \emph{independence property} if there are a formula $\varphi(\bar x,\bar y)$, a model $M$ of $T$, and  elements $\bar a_i$ and $\bar b_S$, for $i\in\bbN$ and $S\subseteq\bbN$ such that $M\models \varphi(\bar a_i,\bar b_S)$ if and only if $i\in S$. The independence property and its negation (NIP) form one of the important dividing lines in model theory.  
The equivalence between \eqref{itm:1} and \eqref{itm:3}  is \cite[Theorem 3]{palyutin1}, but we provide a self-contained proof.

\begin{theorem}\label{thm:palyutinstab}
If $M$ has the simple cover property and $T=\Th(M)$, then the following are equivalent. 
\begin{enumerate}
	\item\label{itm:1} $T$ is unstable,
	\item\label{itm:2} there exists an $h$-formula that has the order property (for $T$),
	\item \label{itm:3} there exist an $h$-formula $\varphi(\bar x,\bar y)$ and $\bar a,\bar b\in M^{|\bar y|}$ such that $\varphi(M,\bar a)$ and $\varphi(M,\bar b)$ intersect nontrivially but are different,
	\item \label{itm:4} there are $k\in \bbN$, an $h$-definable transitive relation $\leqphi$ on $M^k$ which is reflexive on its domain, and elements $\bar a,\bar b\in M^k$ such that $\{\bar a,\bar b\}$ is a $\leqphi$-chain. 
\item \label{itm:5} $T$ has the independence property. 
\end{enumerate}
\end{theorem}

\begin{proof}
We first prove that \eqref{itm:1} and \eqref{itm:2} are equivalent. Clearly~\eqref{itm:2} implies~\eqref{itm:1}. For the direct implication, note that a Boolean combination of stable formulas is again stable, and therefore by Lemma~\ref{lemma:booleanh} if $T$ is unstable and has the simple cover property, then one of the culprit non-stable formulas must be an $h$-formula.

Clearly \eqref{itm:2} implies \eqref{itm:3}. 

\eqref{itm:3} implies \eqref{itm:4}: Let $\varphi(\bar x,\bar y)$ and $\bar a,\bar b\in M^{|\bar y|}$ be witnessing \eqref{itm:3}. Without loss of generality we can assume there is $\bar z\in\varphi(M,\bar b)\setminus\varphi(M,\bar a)$. 

Let $\psi(\bar x,\bar y)=\varphi(\bar x,\bar y)\wedge \varphi(\bar x,\bar b)$; this is an $h$-formula. Notice that $\emptyset\neq \psi(M,\bar a)\subsetneq\psi(M,\bar b)$.
Let
\[
\theta(\bar x,\bar y)=\exists \bar z\,\psi(\bar z,\bar x)\wedge\forall\bar z(\psi(\bar z,\bar x)\Rightarrow \psi(\bar z,\bar y)),
\]
and notice that $\theta(\bar x,\bar y)$ holds if and only if $\emptyset\neq \psi(M,\bar x)\subseteq\psi(M,\bar y)$. Also, $\theta$ is an $h$-formula. The relation $\leqphi$ defined by setting $\bar x\leqphi\bar y$ if and only if $\theta(\bar x,\bar y)$ is transitive and reflexive on its domain and, together with the chain $\{\bar a, \bar b\}$, witnesses \eqref{itm:4}, 

\eqref{itm:4} implies \eqref{itm:2}: By Lemma~\ref{L.dichotomy}, if $\varphi(\bar x, \bar y)$ is an $h$-formula that defines a transitive relation which is reflexive on its domain and has height at least $2$, then $\varphi(\bar x, \bar y)$ has the order property. 

Since \eqref{itm:5} implies \eqref{itm:1}, it remains to prove that \eqref{itm:1} implies \eqref{itm:5}. 
Suppose that $T$ is unstable. We will prove that it has the independence property. By Lemma~\ref{L.scp}, $M$ has the simple cover property. Hence Theorem~\ref{thm:palyutinstab} implies that we can find an $h$-formula $\varphi(x,y)$ and $a$ and $b$ in $M$ such that
\[
M\models \varphi(a,a)\wedge\varphi(a, b)\wedge\varphi( b, b)\wedge\neg\varphi(b,a).
\]
(For convenience, we assume that $\varphi$ is a binary formula).
Lift $a$ and $b$ to $\tilde a=(\tilde a_i)$ and $\tilde b=(\tilde b_i)$ in $\prod_iM_i$. Since $\varphi$ is an $h$-formula, then
\[
X:=\{i\mid M_i\models \varphi(\tilde a_i,\tilde a_i)\wedge \varphi(\tilde a_i,\tilde b_i)\wedge \varphi(\tilde b_i,\tilde a_i)\}
\]
has complement in $\cI$. Let 
\[
Y=\{i\mid M_i\models\neg\varphi(\tilde b_i,\tilde a_i)\}\cap X.
\]
Since $\varphi$ is an $h$-formula, $Y$ is $\cI$-positive. As $\cP(I)/\cI$ is atomless, there are disjoint $\cI$-positive sets $Y_n\subseteq Y$, for $n\in\bbN$. Define $c(n)$, for $n\in\bbN$, by
\[
c(n)_i=\begin{cases}\tilde b_i&\text{ if } i\in Y_n\\
	\tilde a_i&\text{ otherwise} 
\end{cases}.
\]
If $S\subseteq\bbN$, let
\[
d(S)_i=\begin{cases}\tilde b_i&\text{ if } i\in Y_n, n\in S\\
	\tilde a_i&\text{ otherwise}
\end{cases}.
\]
By $\pi\colon\prod_iM_i\to M$ we denote the quotient map. If $n\in S$, then $\{i\mid M_i \vDash \varphi(c(n)_i,d(S)_i)\}\supseteq X$, and therefore 
\[
M\models \varphi(\pi(c(n)),\pi(d(S))).
\]
If $n\notin S$, we have that $Y_n\subseteq \{i\mid M_i\models\neg \varphi(c(n)_i,d(S)_i)\}$. Since $Y_n$ is $\cI$-positive, and $\varphi$ is an $h$-formula, this implies that 
\[
M\models\neg\varphi(\pi(c(n)),\pi(d(S)).
\]
This concludes the proof.
\end{proof}

\section{Failure of saturation}

We will now prove a refinement of Theorem~\ref{T.non-saturated}. If $\cI$ is an ideal on a set $I$ and $A\subseteq I$, then $\cI\rs A$ denotes the ideal $\cI\cap \cP(A)$ on $A$. 

\begin{theorem}\label{T.non-saturated++}
Suppose that $\cI$ is an ideal on a set $I$ such that the Boolean algebra $\cP(I)/\cI$ is atomless and that $\kappa$ is an infinite cardinal. 
\begin{enumerate}
\item \label{1.++} Suppose that $\cP(I)/\cI$ is not $\kappa$-saturated. 
For any structure $M$ in a countable language, if the theory of the reduced power $M^I/\cI$ is unstable then $M^I/\cI$ is not $\kappa$-saturated. 
\item \label{2.++} Suppose that $\cP(A)/(\cI\rs A)$ is not $\kappa$-saturated for any $\cI$-positive~$A$. 
For any choice of structures $M_i$, for $i\in I$, such that the theory of $M=\prod_i M_i/\cI$ is unstable, $M$ is not $\kappa$-saturated. 
\end{enumerate}
\end{theorem}

Part \eqref{1.++} of Theorem~\ref{T.non-saturated++} is a strong converse to  \cite[Theorem~0.2]{shelah1972filters}, which asserts that for every uncountable cardinal $\kappa$ and every ideal $\cI$ on a set $I$ the following holds.  Every reduced product $\prod_i M_i/\cI$ is $\kappa$-saturated if and only if (i) $\cP(I)/\cI$ is $\kappa$-saturated, (ii) $I$ can be covered by countably many sets in $\cI$, and (iii) the dual filter $\cI^*$ is $\lambda$-good for every $\lambda<\kappa$ (for the definition of $\lambda$-good see \cite{shelah1972filters} or \cite{ChaKe}; we will not need it here). 

The following is a generalization of Theorem~\ref{T.non-saturated} to all analytic ideals. 

\begin{corollary}\label{T.non-saturated+}
Let $\cI$ be an analytic ideal on $\bbN$ which includes $\Fin$. If a reduced product $M=\prod_n M_n/\cI$ has unstable theory, then $M$ is not $\aleph_2$-saturated. In particular, if $\CH$ fails then $M$ is not saturated.
\end{corollary}

\begin{proof} 
Since a nonprincipal ultrafilter on $\bbN$ cannot be analytic (e.g., by the Baire Category Theorem), $\cP(\bbN)/\cI$ is atomless. Fix an $\cI$-positive set $A$. Since $\cI$ is analytic, $\cI\rs A$ is analytic in $\cP(A)$, and therefore we can fix (by \cite[Corollary~14]{To:Gaps}) an $(\aleph_1, \aleph_1)$-gap $(A_\alpha,B_\alpha)_{\alpha< \aleph_1}$ in $\cP(A)/\cI \rs A$. Hence $\cP(A)/\cI \rs A$ is not $\aleph_2$-saturated. Since $A$ is arbitrary, the result follows from Theorem~\ref{T.non-saturated++}. 
\end{proof}

In case of a class of ideals that includes the asymptotic density zero ideal, $\cZ_0=\{A\subseteq\bbN\mid \lim\sup_n\frac{|A\cap [0,n]|}{n}=0\}$, we have a sharpening of Corollary~\ref{T.non-saturated+}.
The proof of Corollary~\ref{C.P-ideal} below uses Theorem~\ref{T.layeredcntblysat} which is still ahead of us; however there is no problem since this theorem was already known and its proof does not use Corollary~\ref{C.P-ideal}. 

\begin{corollary}\label{C.P-ideal}
If $\cI$ is an analytic P-ideal, and the theory of $M^I/\cI$ is unstable, then the following are equivalent. 
\begin{enumerate}
\item $\cI$ is an $F_\sigma$ ideal.
\item $M$ is $\aleph_1$-saturated. 
\end{enumerate}
In particular, if $\prod_n M_n/\cZ_0$ has unstable theory, then it is not $\aleph_1$-saturated. 
\end{corollary}

\begin{proof}
If $\cI$ is $F_\sigma$, then it is layered (see Definition~\ref{D.layered}, or \cite{Fa:CH}) and therefore $M$ is $\aleph_1$-saturated by Theorem~\ref{T.layeredcntblysat}. (In this implication we did not use the assumption that $\cI$ is a P-ideal.)

Assume $\cI$ is a P-ideal, but not $F_\sigma$. By Theorem~\ref{T.non-saturated++} it suffices to prove that $\cP(\bbN)/\cI$ is not $\aleph_1$-saturated. This was stated in \cite[p. 14]{Fa:AQ}, but we include a slightly more direct proof. By \cite[Theorem~3.3]{Sol:Analytic}, $\Fin^\omega\leq_f \cI$ (in terms of \cite{Fa:AQ}, $\OFin\leqRB \cI$), meaning that there exists a finite-to-one function $f\colon \bbN\to \bbN^2$ such that $f^{-1}[X]\in \cI$ if and only if $\forall m \exists n\forall k>n (m,k)\notin X$. The sets $A_m=f^{-1}(m\times \bbN)$, for $m\in \bbN$, and $B=\bbN$ form an $\omega$-gap in $\cP(\bbN)/\cI$.

The last sentence follows because $\cZ_0$ is an analytic P-ideal that is not $F_\sigma$ (\cite{Sol:Analytic}). 
\end{proof} 

A consequence of Theorem~\ref{T.non-saturated++} \eqref{2.++} analogous to Corollary~\ref{C.P-ideal} concerned with saturation of a reduced product instead of a reduced power can be proven by the analogous proof. 

\subsection{The proof of Theorem~\ref{T.non-saturated++}}

Both parts of Theorem~\ref{T.non-saturated++} are proven by injecting a gap or a limit into $(M,\leq_\varphi)$ for some $h$-formula $\varphi$. The formula is provided by the assumed unstability of the theory (Theorem~\ref{thm:palyutinstab}), while the gap (or limit) is provided by \cite[Theorem 2.7]{Mija.Boolean} which asserts that for an infinite cardinal $\mu$, an atomless Boolean algebra is $\mu$-saturated if and only if it has no $(\kappa,\lambda)$-gaps or $\kappa$ limits for all cardinals $\kappa,\lambda<\mu$ (this is condition $H_\kappa$ defined in \cite[p. 178]{Mija.Boolean}). 

Recall (Definition~\ref{Def.Gaps}) that even if a binary relation $\leqphi$ on some $M$ is not a partial ordering, one can have gaps and limits in $(M,\leqphi)$. 
 
\begin{lemma}\label{L.gappreserving} 
Suppose that $\cI$ is an ideal on a set $I$ and that $M=\prod_i M_i/\cI$ is a reduced product whose theory $T$ is unstable. Then there are a formula $\varphi(\bar x,\bar y)$ of the language of $T$, an $\cI$-positive $A\subseteq I$, and an order embedding
\[
f\colon (\mathcal P(A)/(\cI \rs A), \subseteq^\cI)\to (M,\leqphi)
\]
sending gaps to gaps and limits to limits. 
\end{lemma}

\begin{proof}
By Theorem~\ref{thm:palyutinstab}, there are an $h$-formula $\varphi(x,y)$ which defines a transitive relation $\leqphi$ which is reflexive on its domain, and elements $a$ and $b$ in $M$ such that $\{a,b\}$ is a $\leqphi$-chain. We may assume that $a$ and $b$ are of arity 1 (see \S\ref{S.arity}). 

Denote by $\pi\colon\prod_i M_i\to M$ the quotient map, and let $\tilde a,\tilde b\in\prod_i M_i$ be such that $a=\pi(\tilde a)$ and $b=\pi(\tilde b)$. Since $\leqphi$ is transitive, there is no $c\in M$ such that $b\leqphi c$ and $c\leqphi a$. Therefore, by Theorem~\ref{T.Palyutin}, the set 
\begin{equation*}\label{eq.A}
A=\{i\in I \mid M_i\models \nexists c\, (\tilde b_i\leqphi c\wedge c\leqphi \tilde a_i)\}
\end{equation*}
is $\cI$-positive. If $B\subseteq A$, let $\chi_B\in \prod_i M_i$ be given by
\[
(\chi_B)_i=\begin{cases} \tilde b_i &\text{if } i\in B,\\
\tilde a_i& \text{otherwise}
\end{cases}.
\]
Define $f\colon (\mathcal P(A)/(\cI \rs A), \subseteq^\cI)\to (M,\leqphi)$ by (writing $[B]_\cI$ for the $\cI$-equivalence class of $B\subseteq I$)
\[
f([B]_\cI)=\pi(\chi_B). 
\]
This is clearly an order embedding, and we will prove that it sends gaps to gaps and limits to limits. 

\begin{claim}\label{c.gappreserving} 
For all $B\subseteq A$ and $\tilde x\in\prod_i M_i$, writing $C_{\tilde x}=\{i\mid \tilde b_i\leqphi \tilde x_i\}$ we have that
\begin{itemize}
\item if $\pi(\chi_B)\leqphi\pi(\tilde x)$ then $B\subseteq^\mathcal I C_{\tilde x}$ and 
\item if $\pi(\tilde x)\leqphi\pi(\chi_B)$ then $C_{\tilde x}\subseteq^\mathcal I B$.
\end{itemize}
\end{claim}
\begin{proof}
Since $\leqphi$ is $h$-definable, Theorem~\ref{T.Palyutin} implies that $\pi(\chi_B)\leqphi\pi(\tilde x)$ if and only if $\{i\mid (\chi_B)_i\nleq_\varphi \tilde{x}_i\}\in \mathcal I$. Since $i\in B$ implies $(\chi_B)_i=\tilde b_i$, for all but $\mathcal I$ many $i\in B$ we get that $ \tilde b_i\leqphi \tilde x_i$, i.e., $i\in C_{\tilde x}$.
 
The second statement is proved by contrapositive: for every $i\in C_{\tilde x}\setminus B$ we have that $ \tilde b_i\leqphi \tilde{x}_i$. As $C_{\tilde x}\subseteq A$, we have $\tilde x_i\nleq_\varphi \tilde a_i$. Hence, if $C_{\tilde x}\setminus B$ is $\mathcal I$-positive, so is the set $\{i\mid \tilde x_i\nleq_\varphi(\chi_B)_i\}$, and therefore $\pi(\tilde x)\nleq_\varphi \pi(\chi_B)$. 
\end{proof}

Applying the claim to $\chi_B$ and $\chi_C$, in case $B$ and $C$ are subsets of $A$, gives that $B\subseteq^\cI C$ if and only if $\pi(\chi_B)\leqphi \pi(\chi_C)$, and therefore this proves that $f$ is an order embedding.

To show that $f$ is gap-preserving, consider cardinals $\kappa$ and $\lambda$ and a $(\kappa,\lambda)$-gap $\{A_\alpha,B_\beta\}$, for $\alpha<\kappa$ and $\beta<\lambda$ in $\mathcal P(A)/(\cI \rs A)$. Since $f$ is an order embedding, for all $\alpha\leq\beta<\kappa$ and $\gamma\leq\delta<\lambda$ we have 
\[
\pi(\chi_{A_\alpha})\leqphi\pi(\chi_{A_\beta})\leqphi\pi(\chi_{B_\gamma})\leqphi\pi(\chi_{B_\delta}).
\]
Suppose there is $d\in M$ such that $\pi(\chi_{A_\alpha})\leqphi d\leqphi\pi(\chi_{B_\beta})$ for all $\alpha<\kappa$ and $\beta<\lambda$. Then Claim~\ref{c.gappreserving} applied to a lift $\tilde d$ of $d$ implies that $C_{\tilde d}$ splits the gap in $\cP(A)/(\cI \rs A)$, a contradiction.

To show that $f$ is limit-preserving, suppose that $A_\alpha$, $B$, for $\alpha<\kappa$, is a limit in $\P(I)/\cI$ and $d$ is such that $\pi(\chi_{A_\alpha})\leqphi d\leqphi\pi(\chi_{B})$ for all $\alpha<\kappa$. Then Claim~\ref{c.gappreserving} implies that (for a lift $\tilde d$ of $d$) we have $A_\alpha\subseteq^{\cI} C_{\tilde d}\subseteq^{\cI} B$ for all $\alpha$. Since $B$ is the limit of $(A_\alpha)$, this implies $B\subseteq^{\cI} C_{\tilde d}$ and therefore $B\Delta C_{\tilde d}\in \cI$, $\Delta$ being the symmetric difference. This implies that $d\leqphi \pi(\chi_B) $ and $\pi(\chi_B)\leqphi d$. 

The proof that $f$ preserves downwards directed limits is analogous, and this completes the proof of the lemma. 
\end{proof}

We are ready to put together the pieces and conclude the proofs of the main results of this section. 

\begin{proof}[Proof of Theorem~\ref{T.non-saturated++}]
\eqref{2.++} Fix an ideal $\cI$ on a set $I$ and an infinite cardinal $\kappa$ such that $\cP(I)/\cI$ is atomless and $\cP(A)/(\cI \rs A)$ is not $\kappa$-saturated whenever $A$ is $\cI$-positive. Fix struc\-tu\-res~$M_i$, for $i\in I$, of the same signature such that $M=\prod_i M_i/\cI$ has unstable theory, $T$. Fix an $\cI$-positive $A\subseteq I$ and 	$\varphi(\bar x,\bar y)$ of the language of $T$ as provided by Lemma~\ref{L.gappreserving}, together with an order embedding
\[
f\colon (\mathcal P(A)/(\cI \rs A), \subseteq^\cI)\to (M,\leqphi)
\]
 that sends gaps to gaps and limits to limits. 	


By hypotheses, the algebra $\cP(A)/(\cI \rs A)$ is atomless and not $\kappa$-saturated, hence by \cite[Theorem 2.7]{Mija.Boolean} there are $\lambda,\mu<\kappa$ such that $\cP(A)/(\cI \rs A)$ has a $(\kappa,\lambda)$-gap or a $\kappa$ limit. In the former case, let $A_\alpha$, $B_\beta$, $\alpha<\kappa$, $\beta<\lambda$ be a gap in $\cP(A)/(\cI\rs A)$. Then the type whose conditions are $f([A_\alpha]_\cI)\leq_\varphi x$, $x\leq_\varphi f([B_\beta]_\cI)$ for $\alpha<\kappa$ and $\beta<\lambda$ is consistent but not realised in $M$. In the latter case, assume that $A_\alpha$, $B$ is a $\kappa$-limit. Then the type whose conditions are $f([A_\alpha]_\cI)\leq_\varphi x$ for all $\alpha$, $x\leq_\varphi f([B_\beta]_\cI)$, and $f([B_\beta]_\cI)\nleq_\varphi x$ is consistent but not realised in $M$. If the limit is of the form $A$, ${B_\alpha}$, $\alpha<\kappa$, the proof is analogous. This concludes the proof. 

The proof of \eqref{1.++} is contained in the provided proof of \eqref{2.++}. 
\end{proof}

\section{Ultrapowers}
Armed with our results, we improve \cite[Theorem~C(1)]{farah2022betweenII}. 

\begin{theorem} \label{T.nonisomorphic} 
There is a forcing extension in which for every sequence~$M_n$, for $n\in \bbN$, of structures in a countable language, if the reduced product $\prod_n M_n/\Fin$ has unstable theory, then it is not isomorphic to any ultrapower associated with a nonprincipal ultrafilter on~$\bbN$.
\end{theorem}

\begin{proof}
We need to prove that there exists a forcing extension in which for every countable language~$\cL$ and $\cL$-structures $N$ and $M_n$, for $n\in \bbN$, and for every nonprincipal ultrafilter $\cU$ on $\bbN$ the ultrapower $N^\bbN/\cU$ is not isomorphic to the reduced product $\prod_n M_n/\Fin$, provided that the theory of the latter is unstable. There is a poset $E$, constructed by Galvin, such that $E$ has no infinite chains but for every linear ordering $\bbL$ that has neither $\aleph_2$-chains nor $\aleph_2^*$-chains, there is no strictly increasing map from $E$ to $\bbL$ (\cite[Theorem~3.2]{Fa:Embedding}). In the forcing extension described in \cite[\S 5]{farah2022betweenII} which consists of the Levy collapse of $\fc$ to $\aleph_1$ followed by a ccc forcing, $E$ embeds into $\prod_n (n,<)/\Fin$ and there are neither $\aleph_2$-chains nor $\aleph_2^*$-chains in $\prod_n L_n/\Fin$ for any sequence of countable partial orders $L_n$, for $n\in \bbN$. 

This embedding $\Phi\colon E\to \prod_n (n,<)/\Fin$ has an additional property that if $\cU$ is a nonprincipal ultrafilter on $\bbN$ and $\pi_\cU\colon \prod_n (n,<)/\Fin\to \prod_n (n,<)/\cU$ is the quotient map, then $\pi_\cU\circ \Phi[E]$ has an $\aleph_2$-chain or an $\aleph_2^*$-chain. This is \cite[Proposition 4.15]{farah2022betweenII}, with $(A_n,\triangleleft_\varphi)=(n,<)$. 

Towards obtaining a contradiction, assume that $N^\bbN/\cU$ is isomorphic to $\prod_n M_n/\Fin$. 
By \L o\' s's Theorem, $N$ has a $\leq_\varphi$-chain of length $n$ for every $n$. Therefore the ultrapower $\prod_\cU (n,<)$ is isomorphic to a substructure of $\prod_\cU (N,\leq_\varphi)/\cU$. By the previous paragraph, there is an embedding $\Phi\colon E\to \prod_n (n<)/\Fin$. The image of $E$ under the composition $\pi_\cU\circ \Phi$ (where $\pi_\cU\colon \prod_n (n,<)/\Fin\to \prod_n (n,<)/\cU$ is the quotient map) has an $\aleph_2$-chain or an $\aleph_2^*$-chain. 
Since $\prod_\cU (n,<)$ embeds in $(N,\leq_\varphi)^\bbN/\cU$, we get an $\aleph_2$-chain or an $\aleph_2^*$-chain in $(N,\leq_\varphi)^\bbN/\cU$.   Since $\varphi$ is definable by an $h$-formula and $N^\bbN/\cU$ is isomorphic to $\prod_n M_n/\Fin$, we have that 
\[
(N,\leq_\varphi)^\bbN/\cU\cong (N^\bbN/\cU,\leq_\varphi)\cong (\prod_n M_n/\Fin,\leq_\varphi)\cong (\prod_nM_n,\leq_\varphi)/\Fin.
\]
This gives an $\aleph_2$-chain, or an $\aleph_2^*$-chain in $(\prod_nM_n,\leq_\varphi)/\Fin$ for a sequence of countable partial orders $(M_n,\leq_\varphi)$, a contradiction.
However, as pointed out earlier, there are no long chains in $\prod_n M_n/\Fin$; contradiction. 
\end{proof}

It is not known whether the assertion that a nontrivial reduced product whose theory is unstable, $\prod_n M_n/\Fin$, is isomorphic to a nontrivial ultrapower is equivalent to CH; see \cite[Question~8.1]{farah2022betweenII}.

\section{Saturation and layered ideals} \label{S.SaturationLayered}

In the present section we introduce layered ideals and start the proof of Theorem~\ref{T.cs}, a strengthening of Theorem~\ref{T.saturated}. Our result was inspired by \cite[Theorem~5.4]{FaHaSh:Model2}, and its original (considerably more complicated) proof was based on the proof of this theorem. The following is adapted from \cite{Fa:CH}, where layered ideals on $\bbN$ have been defined; in this case, the assumption that the index-set is covered by countably many sets in the ideal reduces to the standard assumption that the ideal includes $\Fin$. 

\begin{definition}\label{D.layered}
An ideal $\cI$ on a set $I$ is \emph{layered} if $I$ can be covered by a union of countably many sets in $\cI$ and there is a function 
\[
\mu_\cI\colon \cP(I)\to [0,\infty]
\] 
such that $A\subseteq B$ implies $\mu_\cI(A)\leq \mu_\cI(B)$, $\cI=\{A\mid \mu_\cI(A)<\infty\}$, and $A\in \cI^+$ if and only if $\sup\{\mu_\cI(B)\mid B\subseteq A, B\in \cI\}=\infty$. 
\end{definition}

The function $\mu_{\Fin}(A)=|A|$ witnesses that $\Fin$ is layered. (Note however that the ideal $\Fin(\kappa)$ of finite subsets of an uncountable cardinal $\kappa$ is not layered. Also note that in this case $\cP(\kappa)/\Fin(\kappa)$ has an $\omega$-limit, and is therefore not $\aleph_1$-saturated.) All $F_\sigma$ ideals are also layered and so are other known analytic ideals for which the quotient $\cP(\bbN)/\cI$ is $\aleph_1$-saturated (\cite[Proposition 6.6]{Fa:CH}). Also, the quotient $\cP(\bbN)/\cI$ is $\aleph_1$-saturated for every layered ideal $\cI$ on $\bbN$ (\cite[Lemma~6.7]{Fa:CH}). Conjecturally, an analytic ideal $\cI$ is layered if and only if $\cP(\bbN)/\cI$ is $\aleph_1$-saturated.

\begin{theorem} \label{T.cs}
Suppose that $\cI$ is a layered ideal on a set $I$. Suppose that $M_i$, for $i\in I$, are structures in the same countable language, and that the theory of $M=\prod_i M_i/\cI$ is stable. Then $M$ is $\fc$-saturated. In particular, if $|M_i|\leq\mathfrak c$ for each $i$, then $M$ is saturated.
\end{theorem}

The remainder of this section, as well as the entire \S\ref{S.ReducedStable} are devoted to a sequence of lemmas leading to the proof of Theorem~\ref{T.cs}, given in \S\ref{S.ProofsOfSaturation}. The following proposition is all you need to know about layered ideals. 

\begin{proposition}\label{P.L.1}
Suppose $\cI$ is a layered ideal on a set $I$ and $\cF_n\Subset \cP(I)$ for $n\in \bbN$. Then there is a partition $I=\bigsqcup_n I_n$ such that for every $X\subseteq \bbN$ and every $(A_n)_{n\in X}\in \prod_{n\in X} \cF_n$ we have $\bigcup_{n\in X} (A_n\cap I_n)\in \cI^+$ if and only if $(\exists^\infty n\in X) A_n\in \cI^+$. 
 \end{proposition}

\begin{proof}
Fix $\mu_\cI\colon \cP(I)\to [0,\infty]$ and sets $Y_n\in \cI$ such that $I=\bigcup_n Y_n$ as in Definition~\ref{D.layered}. We will find a partition $I=\bigsqcup_n I_n$ that satisfies the following for all $n$ and $A\in \cF_n$. 
\begin{enumerate}
\item\label{1.l} If $n>0$ and $A\in \cI$ then $A\cap I_n=\emptyset$. 
\item \label{2.l} If $A\in \cI^+$ then $\mu_\cI(A\cap I_j)\geq j$ for all $j\geq n$. 
\item \label{3.l} $I_n\in \cI$ and $\bigcup_{j\leq n} I_j\supseteq \bigcup_{j\leq n} Y_n$. 
\end{enumerate} 
We will describe the recursive construction of the sequence $(I_n)$ with the required properties.

Set $I_0=Y_0\cup \bigcup\{A\mid A\in (\cF_0\cup \cF_1)\cap \cI\}$. Assume that $n\geq 1$ and disjoint $I_j\in \cI$, for $j<n$, had been chosen and they satisfy \eqref{1.l}, \eqref{2.l}, and \eqref{3.l}. Set $I^0_n=Y_n\cup\bigcup\{A\mid A\in \cF_{n+1}\cap \cI\}$. Note, $I_n^0\in\cI$.


For $B\in \bigcup_{j\leq n}\cF_j\cap \cI^+$, find $W_B\in \cI$ with $W_B\subseteq B\setminus \bigcup_{j<n}I_j$ such that $\mu_{\cI}(W_B)\geq n$. Set 
\[
I_n=(I^0_n\cup \bigcup\{W_B\mid B\in \bigcup_{j\leq n}\cF_j\cap \cI^+\})\setminus\bigcup_{j<n}I_j.
\] 
As a union of finitely many sets in~$\cI$, $I_n$ belongs to $\cI$, and this defines a partition $I=\bigsqcup_n I_n$ because the sets $I_n$ are disjoint and $\bigcup_n I_n\supseteq \bigcup_n Y_n$ by construction. The properties \eqref{1.l} and \eqref{2.l} also obviously hold. 

Assume that $X\subseteq \bbN$ and $(A_n)_{n\in X}\in \prod_{n\in X} \cF_n$. If there is $m$ such that $A_n\in \cI$ for all $n>m$, then $\bigcup_{n\in X} (A_n\cap I_n)\subseteq \bigcup_{n\leq m} I_n$ belongs to $\cI$. Now assume that the set $Y$ of all $n\in X$ such that $A_n\in \cI^+$ is infinite. Then $\mu_{\cI}(A_n\cap I_n)\geq n$ for all $n\in Y$, and therefore $\mu_{\cI}(\bigcup_{n\in X} (A_n\cap I_n))=\infty$ and $\bigcup_{n\in X} (A_n\cap \cI_n)\in \cI^+$ as required. 
\end{proof}

The following is such a simple consequence of Proposition~\ref{P.L.1} that it is easier to prove it from scratch. 

\begin{lemma} \label{L.atomless}
If $\cI$ is a layered ideal on $I$ and $A\in \cI^+$ then $\cP(\bbN)/\Fin$ embeds into $\cP(A)/(\cI\rs A)$. In particular, the Boolean algebra $\cP(I)/\cI$ is atomless. 
\end{lemma}

\begin{proof}
Recursively find $X_n\subseteq A\setminus \bigcup_{j<n} X_j$ for $n\in \bbN$ such that $j\leq \mu_\cI(X_j)<\infty$. Then the function $\Phi\colon \cP(\bbN)\to \cP(A)$ defined by $\Phi(Y)=\bigcup_{j\in Y} X_j$ is a lifting of an injective homomorphism from $\cP(\bbN)/\Fin$ into $\cP(A)/\cI$. 
\end{proof}

\section{Reduced products with stable theory}\label{S.ReducedStable}

We continue the proof of Theorem~\ref{T.cs} by analyzing types over a reduced product with stable theory. 
Suppose that $\varphi(x,\bar y)$ is an $h$-formula. Consider the formulas
\begin{align*}
D_\varphi(x)&=\exists \bar y\,\varphi(x,\bar y),\\
E_\varphi(x,x')&=D_\varphi(x)\land \forall \bar y\,(\varphi(x,\bar y)\rightarrow \varphi(x',\bar y)), \\
E_\varphi’(x,x’)&=E_\varphi(x,x’)\wedge E_\varphi(x’,x).
\end{align*}

\begin{lemma}\label{L.er-h}
Suppose that $\cI$ is an ideal on a set $I$, $M=\prod_i M_i/\cI$, and $\varphi(x,\bar y)$ is an $h$-formula. Denote by $\pi\colon\prod_iM_i\to M$ the quotient map. Then $D_\varphi$, $E_\varphi$ and $E’_\varphi$ are $h$-formulas, and the following holds:
\begin{enumerate}
\item $E_\varphi^M$ defines an equivalence relation on $D_\varphi^M$ if and only if
\[
\{i\mid E_\varphi^{M_i}\text{ does not define an equivalence relation on }D_\varphi^{M_i}\}\in\cI;
\] 
\item if $E_\varphi^M$ defines an equivalence relation on $D_\varphi^M$, then for all $c,d$ in $\prod_iM_i$ we have the following. 
\begin{enumerate}
\item $\pi(d)\in D_\varphi^M$ if and only if $\{i\mid d_i\notin D_\varphi^{M_i}\}\in\cI$, and 
\item $M\models E_\varphi(\pi(c),\pi(d))$ if and only if 
\[
\{i\mid M_i\models \neg E_\varphi(c_i, d_i)\}\in\cI.
\] 
\end{enumerate}
\end{enumerate} 
\end{lemma}

\begin{proof} 
From Definition~\ref{Def.h-formulas} it is clear that if $\varphi$ is an $h$-formula then so are $D_\varphi$, $E_\varphi$ and $E’_\varphi$. It is also clear from the definitions that $E_\varphi$ defines a transitive relation on $D_\varphi$, and $E_\varphi’$ defines the symmetrization of $E_\varphi$. Thus $E_\varphi$ defines an equivalence relation if and only if this relation coincides with the one defined by $E_\varphi’$. The thesis then follows from Theorem~\ref{T.Palyutin}.
\end{proof}

In what follows, `type' is a consistent set of formulas (\S\ref{S.Types}) and we will only consider $1$-types (see \S\ref{S.arity}). 

A proof of the following has been extracted from the proof of \cite[Theorem~4]{palmgren}, where its variant for structures with simple cover property and not necessarily stable theory was used to give a simple proof of $\aleph_1$-saturation of reduced products modulo $\Fin$. 
If $N$ is a subset of $M$ and $\bt$ is a 1-type over $N$, we call $\bt$ $N$-complete if for every formula $\varphi(x,\bar a)$ with parameters in $N$, the type $\bt$ contains either $\varphi(x,\bar a)$ or its negation $\neg\varphi(x,\bar a)$.

\begin{lemma}\label{L.type}
	Let $M$ be a structure with the simple cover property and stable theory, let $N \subseteq M$. Suppose that $\bt$ is an $N$-complete 1-type over $N$. Then there are disjoint sets of $h$-formulas $A_\bt$ and $B_\bt$ with at most $x$ as a free variable, $\bar a_\varphi$ in $N$ for all $\varphi\in A_\bt$ (of the appropriate sort), and, for each $\psi\in B_{\bt}$, a set of tuples $\cX_\psi$ in $N$ of the appropriate sort such that 
	\[
	\{\varphi(x,\bar a_\varphi)\mid \varphi\in A_\bt\}\cup 
	\{\lnot \psi(x,\bar b)\mid\psi\in B_\bt, \bar b\in \mathcal X_\psi\}
	\]
	axiomatizes $\bt$. 
\end{lemma}

\begin{proof} 
By Lemma~\ref{lemma:booleanh} and the simple cover property, every formula in $\bt(x)$ is a Boolean combination of $h$-formulas. Therefore $\bt$ is axiomatized by a set $\Theta(\bt)$ of finite disjunctions of $h$-formulas and negations of $h$-formulas with parameters in~$N$. By $N$-completeness, of the type $\bt$, we can even assume that $\Theta(\bt)$ contains only $h$-formulas and negations of $h$-formulas. Let 
\[
A_\bt=\{\varphi(x,\bar y)\mid \varphi\text{ is an $h$-formula and 
$\varphi(x,\bar a)\in \Theta(\bt)$ for some $\bar a$}\}.
\]
Since the theory of $M$ is stable, by Theorem~\ref{thm:palyutinstab}, for every $h$-formula $\varphi(x,\bar y)$, $E_\varphi$ defines an equivalence relation on $D_\varphi^M$, as $\bar y$ varies among tuples of the appropriate sort. Therefore for $\varphi\in A_\bt$ we can choose $\bar a_\varphi$ to be an arbitrary tuple in $M$ such that $\varphi(x,\bar a_\varphi)\in\Theta(\bt)$, and (modulo the theory of $M$) this implies all other instances $\varphi(x,\bar b)$ that belong to $\Theta(\bt)$. 
 
Consider $C_\bt$ to be the set of all $h$-formulas $\psi(x,\bar b)$ whose negation belongs to $\Theta(\bt)$. For any such $\psi$, let $\mathcal X_\psi$ be the set of tuples $\bar b$ such that $\neg\psi(x,\bar b)\in\Theta(\bt)$. It is possible that $C_\bt$ and $A_\bt$ intersect (for example, perhaps both $x=a$ and $x\neq b$ could belong to $\Theta(\bt)$). Nevertheless, if $\varphi$ belongs to both $A_\bt$ and $B_\bt$ (witnessed by $\bar a_\varphi$ and $\bar b_\varphi$), by consistency, $\varphi(x,\bar a_\varphi)$ implies that $\neg\varphi(x,\bar b_{\varphi})$, and therefore we can remove $\varphi$ from $B_\bt$. Setting $B_\bt=C_\bt\setminus A_\bt$ we get the appropriate sets.
\end{proof}

\begin{definition}\label{D.h-axiomatization}
An \emph{$h$-axiomatization} of a type $\bt$ is a disjoint pair of sets of formulas, $\bt_A$ and $\bt_B$ such that every formula in $\bt_A$ is an $h$-formula, every formula in $\bt_B$ is the negation of an $h$-formula, and $\bt_A\sqcup \bt_B$ axiomatizes $\bt$. 
\end{definition}

Lemma~\ref{L.type}, asserts that every $N$-complete type over a subset $N$ of a structure with the simple cover property and stable theory has an $h$-axiomatization. We can therefore focus on types of the form $\bt_A\sqcup\bt_B$. Confusing a type with its axiomatization is a harmless practice that we will indulge in. 

Given a type $\bt$ of the form $\bt_A\sqcup \bt_B$, $n\geq 1$ and $B_0\Subset B_\bt$, let $\bt[B_0,n]$ be the $n$-type in variables $x_0,\dots, x_{n-1}$ which includes $\bt_{A}(x_i)$ for $i<n$ as well as the conditions $\lnot E_\psi(x_i,x_j)$ for all $\psi\in B_0$ and $i<j<n$. 

\begin{definition}
A type is \emph{2-$h$-maximal} if it has an $h$-axiomatization $\bt_A\sqcup \bt_B$ for which $\bt[\{\psi\},2]$ is consistent for every $\psi\in B_\bt$. Such axiomatization is called \emph{2-$h$-maximal}.

A type is \emph{$h$-maximal} if it has an $h$-axiomatization $\bt_A\sqcup \bt_B$ for which $\bt[B_0,n]$ is consistent for every $B_0\Subset B_\bt$ and every $n\geq 2$. Such axiomatization is called \emph{$h$-maximal}. 
\end{definition}

Equivalently, an $h$-axiomatization $\bt_A\sqcup \bt_B$ is 2-$h$-maximal if there are no $A_0\Subset A_\bt$ and $\psi\in B_\bt$ for which $\bigcap_{\varphi\in A_0}\varphi(M,\bar a_{\varphi})$ is included in a unique $E_\psi$-equivalence class. It is $h$-maximal if every set of the form $\bigcap_{\varphi\in A_0}\varphi(M,\bar a_{\varphi})$ intersects infinitely many equivalence classes $E_\psi$, for $\psi\in B_\bt$. Lemma~\ref{L.maxtype1} and Lemma~\ref{L.partialtype1} below together imply that every type over a model with the simple cover property which has an $h$-axiomatization is $h$-maximal.

\begin{lemma} \label{L.maxtype1} 
Every type $\bt$ with an $h$-axiomatization is 2-$h$-maximal. 
\end{lemma}
\begin{proof} Consider an $h$-axiomatization $\bt_A\sqcup \bt_B$ of $\bt$ with 
\[
\bt_A=\{\varphi(x,\bar a_\varphi)\mid \varphi\in A_\bt\}\text{ and }
\bt_B=\{\lnot \psi(x,\bar b)\mid\psi\in B_\bt, \bar b\in \mathcal X_\psi\}.
\]
Let $C$ be the set of all formulas $\neg\psi(x,\bar y) \in B_\bt$ for which there exist $A_0\Subset A_\bt$ and $\bar b_\psi$ in $M$ of the appropriate sort such that
\begin{equation} \label{eq:n2}
M \vDash (\forall x) \left((\exists \bar{y}) \psi(x,\bar y) \wedge \left(\bigwedge_{\varphi\in A_0}\varphi(x,\bar a_\varphi)\right)\right)\Rightarrow\psi(x,\bar b_\psi).
\end{equation} 
Consider the $h$-axiomatization 
\[
\bt'=\bt_A\sqcup (\bt_B\setminus \{\neg \psi(x,\bar b)\mid \neg\psi(x,\bar y)\in C, \bar b\in\mathcal X_\psi\}).
\] 
It is 2-$h$-maximal: if $\bt'[\{\psi\},2]$ would fail to be consistent for certain $\psi \in B_{\bt'}$, this would imply $\psi \in C$, so $\psi$ couldn't have been in~$B_{\bt'}$. 
Moreover,~$\bt'$ axiomatises $\bt$ because it is included in $\bt$ and every formula removed from~$\bt$ belongs to $\bt_B$ and is by \eqref{eq:n2} implied by $\bt_A$. 
	%
	%
	%
	%
\end{proof}

\begin{lemma}\label{L.partialtype1}
Suppose that a structure $M$ has the simple cover property. Then every $h$-axiomatizable type $\bt$ over $M$ is $h$-maximal. 
\end{lemma}
\begin{proof}
By Lemma~\ref{L.maxtype1} we can fix a 2-$h$-maximal axiomatization $\bt_A\sqcup \bt_B$ of $\bt$. Fix $B_0\Subset B_\bt$. We first prove that $\bt[B_0,2]$ is consistent. Let $F\Subset B_0$. By enlarging $F$ we can assume that there is $A_0\Subset A_\bt$ such that
\[
F=\{\varphi(x_i,\bar a_\varphi)\mid \varphi\in A_0, i<2\}\cup \{\neg E_{\psi}(x_0,x_1) \mid \psi\in B_0\}.
\]
Since each $\bt[\{\psi\},2]$ is consistent for $\psi\in B_0$, we can fix elements $c_{j,\psi}$ in $M$, for $\psi\in B_0$ and $j\in\{0,1\}$, such that 
\[
M\models \bigwedge_{j=0,1}\bigwedge_{\psi\in B_0}\bigwedge_{\varphi\in A_0}\varphi(c_{j,\psi},\bar a_\varphi)\wedge \bigwedge_{\psi\in B_0} \neg E_{\psi}(c_{0,\psi},c_{1,\psi}).
\]
If $\bt[B_0,2]$ is not consistent, then we have 
\[
M\models(\forall x) (\bigwedge_{\varphi\in A_0}\varphi(x,\bar a_\varphi)\Rightarrow (\bigvee_{j=0,1}\bigvee_{\psi\in B_0} E_\psi(x,c_{j,\psi}))).
\]
Since $\psi$ is an $h$-formula and $M$ has the simple cover property, we can find $j\in\{0,1\}$ and $\psi\in B_0$ such that 
\[
M\models(\forall x)(\bigwedge_{\varphi\in A_0}\varphi(x,\bar a_\varphi)\Rightarrow E_\psi(x,c_{j,\psi})).
\]
This contradicts the assumed consistency of $\bt[\{\psi\},2]$.

We now proceed by induction. Fix $n\geq 3$, and $A_0\Subset A$. Since $\bt[B_0,n-1]$ is consistent, we can find $c_0,\ldots,c_{n-2}$ in $\varphi(M,\bar a_{\varphi})$ such that $M\models\neg E_\psi(c_i,c_j)$ for all $\psi\in B_0$ and $i\neq j$. If $\bt[B_0,n]$ is not consistent, as above
\[
M\models(\forall x) (\bigwedge_{\varphi\in A_0}\varphi(x,\bar a_\varphi)\Rightarrow (\bigvee_{j<n-1}\bigvee_{\psi\in B_0} E_\psi(x,c_{j}))).
\]
As above, the simple cover property implies that there is $j<n-1$ and $\psi\in B_0$ such that 
\[
M\models(\forall x) (\bigwedge_{\varphi\in A_0}\varphi(x,\bar a_\varphi)\Rightarrow E_\psi(x,c_{j})), 
\]
a contradiction.
\end{proof}

\section{Proofs of saturation}\label{S.ProofsOfSaturation}
As a warm up for the proof of Theorem~\ref{T.cs} we provide a considerably simpler proof of \cite[Theorem 2.7]{FaSh:Rigidity}, based on the proof of \cite[Theorem~4]{palmgren}, where the analogous statement has been shown for countably generated ideals. For layered ideals see Definition~\ref{D.layered}. Also remember that `type’ means `a consistent set of formulas in $x$’ and that `type over $M$' means `a consistent set of formulas in language expanded by adding constants for elements of $M$' (see \S\ref{S.Types}).

\begin{theorem}\label{T.layeredcntblysat}
Let $\cI$ be a layered ideal on a set $I$. Suppose that $M_i$, for $i\in I$, are structures in the same countable language. Then $\prod_iM_i/\cI$ is $\aleph_1$-saturated.
\end{theorem}

\begin{proof}

If the index set $I$ is countable then by \cite{Fa:CH}, $\cP(\bbN)/\cI$ is $\aleph_1$-saturated and by \cite{pacholski1970countably} so is $\prod_i M_i/\cI$ (the authors of this paper refer to countable compactness; see \S\ref{S.Types} for clarification of terminology). By the latter result, it would suffice to prove that $\cP(I)/\cI$ is $\aleph_1$-saturated, but we include a proof of the full result. 

Let $\bt(x)$ be an $N$-complete type over some countable subset $N$ of $M$.
Since $\cP(\bbN)/\cI$ is atomless, by Lemma~\ref{lemma:booleanh}, every formula in $\bt$ is equivalent to a Boolean combination of $h$-formulas. Therefore $\bt$ has an axiomatization of the form $\{\varphi_n(x,\bar a_n), \neg\psi_n(x,\bar b_n)\}_{ n\in\mathbb{N}}$ where each $\varphi_n$ and each $\psi_n$ is an $h$-formula, and $\bar a_n$ and $\bar b_n$ are parameters in $N$ of the appropriate sort. 
(Note that we are not using Lemma~\ref{L.type} and both $\psi_m=\psi_n$ and $\varphi_m=\varphi_n$ are allowed for different $m$ and $n$.) 
Lift each $\bar a_n$ and $\bar b_n$ to $\tilde a_n$ and $\tilde b_n$ in $\prod_iM_i$. Since $\bt$ is consistent, we can find $c_n\in M$  
which realises the type $\{\varphi_m(x,\bar a_m),\neg\psi_m(x,\bar b_m)\}_{m\leq n}$. 
Lift $c_n$ to $\tilde c_n\in\prod_i M_i$, and let
\[
A^m_n=\{i\in I\mid M_i\models\neg\varphi_m((\tilde c_n)_i,(\tilde a_m)_i)\},
\]
and 
\[
B^m_n=\{i\in I \mid M_i\models \neg\psi_m((\tilde c_n)_i,(\tilde b_m)_i)\}.
\]
Since all $\varphi_m$ and all $\psi_m$ are $h$-formulas, by Theorem~\ref{T.Palyutin} we have $A^m_n\in \cI$ and $B^m_n\in \cI^+$ for all $m\leq n$.

By Proposition~\ref{P.L.1} applied to $\cF_n=\{A^m_n, B^m_n\mid m\leq n\}$ there is a partition $I=\bigsqcup_n I_n$ such that $\bigcup_n (A^m_n\cap I_n)\in \cI$ and $\bigcup_n (B^m_n\cap I_n)\in \cI^+$ for all $m$. Define $\tilde c\in \prod_i M_i$ by 
\[
\tilde c_i=(\tilde c_n)_i\text{ if } i\in I_n. 
\]
Then for every $m$ the set $\{i\mid M_i\models \varphi_m((\tilde c)_i,(\tilde a_m)_i)\}=\bigcup_n (A^m_n\cap I_n)$ is in $\cI$ while $\{i \mid M_i\models \neg\psi_m((\tilde c)_i,(\tilde b_m)_i)\}=\bigcup_n (B^m_n\cap I_n)$ is in $\cI^+$. Thus Theorem~\ref{T.Palyutin} implies that $c=\pi(\tilde c)$ realizes $\bt$, where $\pi\colon\prod_iM_i\to M$ is the quotient map. 
\end{proof}

\begin{proof}[Proof of Theorem~\ref{T.cs}] 
Suppose that $M_i$, for $i\in I$, are structures in the same countable language, $\cI$~is a layered ideal on $I$, and the theory of $M=\prod_i M_i/\cI$ is stable. We need to prove that $M$ is $\fc$-saturated. 
 
Fix a complete type $\bt$ over a subset of $M$ of cardinality $<\mathfrak c$. By Lemmas~\ref{L.type} and ~\ref{L.maxtype1}, and \ref{L.partialtype1}, $\bt$ has a maximal $h$-axiomatization of the form 
\[
\{\varphi_n(x,\bar a_n)\mid n\in \bbN\}\cup \{\neg\psi_n(x,\bar b)\mid n\in \bbN, \bar b \in \mathcal X_n\},
\]
where $\{\varphi_n\}$ and $\{\psi_n\}$ are disjoint sets of $h$-formulas in the free variable $x$. 
By $h$-maximality of $\bt$ and Lemma~\ref{L.partialtype1}, for every $B_0\Subset B_\bt$ and a positive $n\geq 2$, $\mathbf t[B_0,n]$ is consistent. 


As before, we lift $\bar a_n$ to $\tilde a_n\in\prod_iM_i$. We also let $Y_n$ be sets in $\cI$ such that $\bigcup Y_n=I$. 
Since by $h$-maximality for each $B_0\Subset B_\bt$ the type $\bt[B_0,n]$ is consistent, we can find $d_{n,u}\in M$, for $u\in \{0,1\}^n$, such that 
\[
M\models \bigwedge_{j\leq n}\bigwedge_{u\in \{0,1\}^n} \varphi_j(d_{n,u},\bar a_{j}) \wedge \bigwedge_{j\leq n}\bigwedge_{u\neq v}\neg E_{\psi_j}(d_{n,u},d_{n,v}).
\]
Lift $d_{n,u}$ to $\tilde d_{n,u}\in\prod_iM_i$. Then for $m\leq n$ and $u\neq v$ in $\twon$, the set 
\[ 
A_{m,n,u,v}:=\{i\mid M_i\models \neg E_{\psi_m}((\tilde d_{n,u})_i,(\tilde d_{n,v})_i)\}
\]
is $\cI$-positive by Theorem~\ref{T.Palyutin}. (This theorem will be used quite a few times in the following lines.) Set $A_{m,n,u,u}=\emptyset$ for $u\in \{0,1\}^n$. Since each~$E_{\psi_n}$ is an equivalence relation, we have $A_{m,n,u,v} = A_{m,n,v,u} $ for all $m\leq n$ and $u,v$ in $\{0,1\}^n$. For $m\leq n$ and $u\in \{0,1\}^n$ the set 
\[
B_{m,n,u} =\{i \mid M_i\models \lnot \varphi_m(d_{n,u})\}
\]
belongs to $\cI$. Let 
\[
\cF_n=\{A_{m,n,u,v}, B_{m,n,u}\mid m\leq n, u\neq v, u,v\in \{0,1\}^n\}.
\] 
Proposition~\ref{P.L.1} implies that there is a partition $I=\bigsqcup_n I_n$ such that $\bigcup_n (B_{m,n,f\rs n}\cap I_n)\in \cI$ and $\bigcup_n (A_{m,n,f\rs n,g\rs n}\cap I_n)\in \cI^+$ for all $m$ and $f\neq g$ in $\{0,1\}^{\bbN}$. For $f\in \{0,1\}^{\bbN}$, define $\tilde c_f\in\prod_iM_i$ by 
\[
(\tilde c_f)_i=(\tilde d_{n,f\rs n})_i\text{ if } i\in I_n.
\]
Set $c_f=\pi(\tilde c_f)$. For every $f\in \{0,1\}^\bbN$ and $m\in \bbN$ we have that 
\[
\{i\mid M_i\models \lnot \varphi_m(c_{f})\}=\bigcup_n (B_{m,n,f\rs n}\cap I_n)
\] 
belongs to $\cI$, thus $M\models \varphi_m(c_f, a_m)$. For distinct $f$ and $g$ in $\twoN$, we have that 
\[
\{i\mid M_i\models \neg E_{\psi_m}(d_f,d_g)\}=\bigcup_n A_{m,n,f\rs n,g\rs n}, 
\]
is $\cI$-positive and Lemma~\ref{L.er-h} (together with Theorem~\ref{T.Palyutin}) implies $M\models \neg E_{\psi_m}(d_f,d_g)$. Thus for every $m$ and $\bar b\in \mathcal X_m$ there is at most one $f\in \twoN$ such that $M\models\psi_m(c_f,\bar b)$. Hence the set 
\[
Y_m=\{f\in \twoN\mid M\models \psi_m(c_f,\bar b)\text{ for some } \bar b \in \mathcal X_m\}
\] 
has cardinality at most $|\bt|$, and therefore there is $f\in \twoN\setminus\bigcup_{m} Y_m$, and $c_f$ satisfies $\bt$. This concludes the proof.
\end{proof}

\section{Concluding remarks}

\subsection{Existence of saturated models}\label{S.existence}

The foregoing results, which identify stability as the separating factor for (non-)saturation of reduced products, are but one instance of a general tight link between stability and existence of saturated models.
Indeed, it is even consistent (modulo large cardinals) that stable theories are the only theories with uncountable saturated models.
We collect a few well-known related results that apparently haven't been stated in a single convenient place yet.
The standard construction of saturated models requires assumptions on cardinal arithmetic, more precisely a regular cardinal~$\kappa$ such that $2^{<\kappa}=\kappa$ (see the proof of \cite[Proposition~5.1.5]{ChaKe}). However, if $T$ is complete and stable then it has a saturated model; more precisely, it has a saturated model of every cardinality in which it is stable (\cite{harnik1975existence}, also \cite[Theorem~VIII.4.7]{shelah1990classification}). More specifically, if the language of~$T$ is countable and~$A$ is a countable model of $T$, then the ultrapower of~$A$ associated with a nonprincipal ultrafilter on $\bbN$ is saturated by \cite[Theorem~5.6]{FaHaSh:Model2}. 

The question of the existence of saturated models of unstable theories is more interesting. 	 The assertion `some first-order theory in a countable language has no uncountable saturated models’  is, by \cite[Theorem~VIII.4.7]{shelah1990classification},  equivalent to (see the   last paragraph of \cite{halevi2023saturated} for  details)
\begin{enumerate}
	\item [$(\dagger)$] No cardinal is inaccessible and $\mathsf{GCH}$ fails at every cardinal. 
\end{enumerate}   
The large cardinal strength  of this statement lies between the existence of a measurable cardinal $\kappa$ of Mitchell order\footnote{For general background on large cardinals see \cite{Kana:Book} and for the Mitchell ordering  see \cite{mitchell2009beginning}.} $\kappa^{++}$ and one of a Mitchell order $\kappa^{+++}$.  The latter assumption  suffices to obtain ($\dagger$) by \cite{merimovich2007power} (see also 
\cite{gitikmerimovich}). The converse follows from the consistency strength for failure of the Singular Cardinal Hypothesis (\cite{gitik1991strength}). 
  See \cite{schweber} and especially \cite{caicedo} for more details.

All that said, it is well-known that an absoluteness argument implies that for most purposes the assumption that a saturated model exists is innocuous (i.e., removable by standard metamathematical means); see~\cite{halevi2023saturated}. 

\subsection{Saturation of reduced products with stable theory}

By the Feferman--Vaught theorem, the theory of $M=\prod_i M_i/\cI$ can be computed from the relation between the theories of $M_i$ and the structure of the Boolean algebra $\cP(I)/\cI$. In particular, the question whether the theory of $M$ is stable depends only on the theories of the~$M_i$’s. It is however not clear how exactly stability of the theory of $M$ depends on the stability properties of the theories of the factors $M_i$, even when $I=\mathbb N$ and $\cI=\Fin$. For example, we do not know of a convenient characterisation, in terms of $\Th(M_0)$, of those first order structures $M_0$ such that $\Th(M_0^\bbN/\Fin)$ is stable.
The following examples show that, when $\cI$ is any ideal on a set $I$ such that $\cP(I)/\cI$ is infinite, (un-)stability of one coordinate of a pair $(M_0,M_0^I/\cI)$ on itself doesn't give any information about (un-)stability of the other coordinate, since all four possible combinations can occur:

\begin{examples}\label{examples}
\begin{enumerate}[label=(\arabic*), wide, labelindent=0pt]
\item \label{example.1} If some $h$-formula $\varphi$ has the order property in $\Th(M_i)$ for all $i \in S$ with $S \in \cI^+$,  
then $\varphi$ has the order property in $\Th(\prod_i M_i/\cI)$ and the latter is unstable.\footnote{A very special case of this is the robust order property of \cite{farah2022betweenII}.}
\item \label{example.2}  Suppose $T$ is a theory all of whose models have stable theory, which is moreover axiomatized by Horn sentences (such as e.g.\ the theory of abelian groups).  If $M_i$ are all models of $T$, then $M=\prod_i M_i/\cI$ is stable.
\item \label{example.3} There are a language $\calL$ and an $\calL$-structure $M_0$ such that $\Th(M_0)$ is stable but $\Th(M_0^I/\cI)$ is unstable. As in the proof of Theorem~\ref{thm:palyutinstab}, take $M_0$ to be a two-element linear order. Then, if $\cP(I)/\cI$ is infinite, $M_0^I/\cI$ is a partial ordering with infinite chains.\footnote{Alternatively, take $M_0$ to be the two-element Boolean algebra, in which case one obtains $\cP(I)/\cI$.}
\item \label{example.4} Finally, reduced products of unstable structures can become stable. In fact this is even the case for reduced powers: \cite[Example~1.5]{wierzejewski} gives an example of a structure $M_0$ with unstable theory such that $M_0\times M_0$ has stable theory and additionally, for every ideal $\cI$ that is not maximal, the reduced power $M_0^I/\cI$ is elementarily equivalent to $M_0\times M_0$. 
\end{enumerate}
\end{examples} 

Notably, for every infinite cardinal $\kappa$ Shelah characterized ideals $\cI$ such that \emph{every} reduced product associated with $\cI$ is $\kappa$-saturated (\cite{shelah1972filters}, see also \cite{pacholski1970countably} and \cite{shelah2021atomic}).\footnote{Shelah's proof uses the Feferman--Vaught theorem, and it can probably be simplified by using Palmgren's trick as in the proof of Theorem~\ref{T.layeredcntblysat}.} Our Corollary~\ref{T.non-saturated+} (or, in fact, the presence of Hausdorff gaps in analytic quotients, see \cite{To:Gaps}) shows that no analytic ideal on $\bbN$ satisfies these assumptions for $\kappa\geq \aleph_2$, because the reduced products whose theory is unstable cannot be $\aleph_2$-saturated. 

We thus restrict to reduced products whose theory is stable. In Theorem~\ref{T.cs} we proved that the reduced product associated with a layered ideal is $\fc$-saturated if its theory is stable. We conjecture that the assumptions on the ideal $\cI$ in this theorem are far from optimal. In \cite[Theorem~5.6]{FaHaSh:Model2} it has been proven that the ultraproduct of countable structures associated with a nonprincipal ultrafilter on $\bbN$ is saturated if its theory is stable. What about reduced products? 

\begin{question} \label{C.atomless}
Is it true that if $\cI$ is any ideal on $\bbN$ such that $\cI\supseteq \Fin$ and the theory of $M=\prod_i M_i/\cI$ is stable, then $M$ is $\fc$-saturated?
\end{question}

A prominent class of ideals for which we don’t know the answer is the class of density ideals. These are $F_{\sigma\delta}$ ideals constructed from lower semicontinuous submeasures on $\bbN$ (see \cite[\S1.13]{Fa:AQ}), extending ideals known as Erd\"os--Ulam ideals (see for example \cite{JustKr}, \cite{Just:Repercussions}). This class includes the ideal $\cZ_0=\{A\subseteq\bbN\mid \lim\sup_n\frac{|A\cap [0,n]|}{n}=0\}$ of asymptotic density zero subsets of $\bbN$. By Corollary~\ref{T.non-saturated+}, no reduced product over $\cZ_0$ whose theory is unstable is $\aleph_1$-saturated. We do not know whether the conclusion of Theorem~\ref{T.cs} can hold for an ideal $\cI$ on $\bbN$ such that $\cP(\bbN)/\cI$ is not $\aleph_1$-saturated. This may depend on the first-order theories of the structures involved. If the theory of $M$ is $\aleph_1$-categorical then the conclusion of Theorem~\ref{T.cs} holds trivially\footnote{`Trivially’ modulo the landmark Morley’s theorem, which implies that all models of $\Th(M)$ of the same uncountable cardinality are isomorphic (see e.g., \cite[\S 7.1]{ChaKe}).}. This leaves open the possibility that the dividing line for full saturation of reduced products associated with (for example) $\cZ_0$ is provided by a property strictly stronger than stability. 

Here is a less ambitious variant of Question~\ref{C.atomless}. 
\begin{question}\label{Q.2} 
Suppose $\cI$ is an ideal on $\bbN$, and $M_n$ are structures such that $\prod_nM_n/\cI$ has stable theory and is $\aleph_1$-saturated. Is $\prod_nM_n/\cI$ $\fc$-saturated?
\end{question}

The proof of Theorem~\ref{T.cs} can be modified to give a positive answer to Question~\ref{Q.2} in case when $\cI$ has the disjoint refinement property (even if it is not layered). An ideal has the \emph{disjoint refinement property} if every sequence $Z_n$, for $n\in \bbN$, of $\cI$-positive subsets of $I$ has a disjoint refinement: $\cI$-positive and disjoint sets $Y_n\subseteq Z_n$, for $n\in \bbN$. It is not difficult to see that Proposition~\ref{P.L.1} implies that every layered ideal has the disjoint refinement property. As Michael Hru\v s\' ak pointed out, the disjoint refinement property fails for some ideals $\cI$. This is because for every poset $\bbP$ of cardinality not greater than $\fc$ the regular open algebra of $\bbP$ is isomorphic to $\cP(\bbN)/\cI$ for some ideal~$\cI$ (e.g., \cite{MoSo}). If~$\bbP$ is chosen to be a countable atomless poset, then the corresponding quotient $\cP(\bbN)/\cI$ is atomless yet it has a countable, dense (in the forcing sense) subset and it therefore fails the disjoint refinement property.

In \cite{Keisler.Ultraproducts} Keisler defined an order on first-order theories in terms of saturation of ultrapowers. After a long hiatus, in the recent years several spectacular results have been proven about it (\cite{malliaris2013general}, \cite{malliaris2016cofinality}, \cite{malliaris2018keisler}). While an analog of the Keisler order associated with reduced powers with respect to the ideals $\cI$ such that $\cP(I)/\cI$ is atomless can be naturally defined on theories with the simple cover property (i.e., on theories preserved under reduced products,  Corollary~\ref{C.Preservation}), we are not aware of any research in this direction.

We conjecture that results analogous to our main theorems hold in continuous model theory (\cite{BYBHU}; the result of \cite[Theorem~5.6]{FaHaSh:Model2} includes this case). Even proving $\aleph_1$-saturation of reduced products associated with $\Fin$ is considerably more technical in continuous model theory than in the case of discrete logic (see \cite[Theorem~1.5]{FaSh:Rigidity} and \cite[Theorem~16.5.1]{Fa:STCstar} for a simpler proof). These proofs use the Feferman--Vaught theorem for continuous logic (\cite{ghasemi2014reduced}, see also \cite[\S 16.2]{Fa:STCstar}). It is likely that a simpler proof, using continuous versions of the results in \S\ref{S.ReducedStable}, exists. The analog of Palyutin's theory in continuous setting has been developed in~\cite{fronteau2023produits}.

\bibliographystyle{amsplain}
\bibliography{library} 

\providecommand{\bysame}{\leavevmode\hbox to3em{\hrulefill}\thinspace}
\providecommand{\MR}{\relax\ifhmode\unskip\space\fi MR }
\providecommand{\MRhref}[2]{%
  \href{http://www.ams.org/mathscinet-getitem?mr=#1}{#2}
}
\providecommand{\href}[2]{#2}
\begin{thebibliography}{10}

\bibitem{BYBHU}
I.~Ben~Yaacov, A.~Berenstein, C.W. Henson, and A.~Usvyatsov, \emph{Model theory
  for metric structures}, Model Theory with Applications to Algebra and
  Analysis, Vol. II (Z.~Chatzidakis et~al., eds.), London Math. Soc. Lecture
  Notes Series, no. 350, London Math. Soc., 2008, pp.~315--427.

\bibitem{BYU:ContStab}
I.~Ben~Yaacov and A.~Usvyatsov, \emph{Continuous first order logic and local
  stability}, Trans. Amer. Math. Soc \textbf{362} (2010), 5213--5259.

\bibitem{caicedo}
A.E. Caicedo, \emph{What is the consistency strength of generalized failure of
  the continuum hypothesis.}, Mathematics Stack Exchange,
  \url{https://math.stackexchange.com/q/3186971} (version:2019-05-14).

\bibitem{ChaKe}
C.~C. Chang and H.~J. Keisler, \emph{Model theory}, third ed., Studies in Logic
  and the Foundations of Mathematics, vol.~73, North-Holland Publishing Co.,
  Amsterdam, 1990.

\bibitem{de2023trivial}
B.~De~Bondt, I.~Farah, and A.~Vignati, \emph{Trivial automorphisms of reduced
  products}, arXiv:2307.06731.

\bibitem{Fa:Embedding}
I.~Farah, \emph{Embedding partially ordered sets into $\omega^\omega$}, Fund.
  Math. \textbf{151} (1996), 53--95.

\bibitem{Fa:AQ}
\bysame, \emph{Analytic quotients: theory of liftings for quotients over
  analytic ideals on the integers}, Mem. Amer. Math. Soc. \textbf{148} (2000),
  no.~702, xvi+177.

\bibitem{Fa:CH}
\bysame, \emph{How many {B}oolean algebras {$\cP(\bbN)/\cI$} are there?},
  Illinois Journal of Mathematics \textbf{46} (2003), 999--1033.

\bibitem{Fa:STCstar}
\bysame, \emph{Combinatorial set theory and \cstar-algebras}, Springer
  Monographs in Mathematics, Springer, 2019.

\bibitem{farah2022betweenI}
\bysame, \emph{Between reduced powers and ultrapowers}, Journal of the European
  Mathematical Society \textbf{25} (2022), no.~11, 4369--4394.

\bibitem{farah2022corona}
I.~Farah, S.~Ghasemi, A.~Vaccaro, and A.~Vignati, \emph{Corona rigidity},
  arXiv:2201.11618.

\bibitem{FaHaSh:Model1}
I.~Farah, B.~Hart, and D.~Sherman, \emph{Model theory of operator algebras {I}:
  Stability}, Bull. London Math. Soc. \textbf{45} (2013), 825--838.

\bibitem{FaHaSh:Model2}
\bysame, \emph{Model theory of operator algebras {II}: Model theory}, Israel J.
  Math. \textbf{201} (2014), 477--505.

\bibitem{FaSh:Rigidity}
I.~Farah and S.~Shelah, \emph{Rigidity of continuous quotients}, J. Math. Inst.
  Jussieu \textbf{15} (2016), no.~01, 1--28.

\bibitem{farah2022betweenII}
\bysame, \emph{Between reduced powers and ultrapowers, {II}.}, Trans. Amer.
  Math. Soc \textbf{375} (2022), 9007--9034.

\bibitem{fronteau2023produits}
I.~Fronteau, \emph{Produits r\' eduits en logique continue}, Master's thesis,
  Universit\' e Paris Cit\' e, 2023.

\bibitem{galvin1970horn}
F.~Galvin, \emph{Horn sentences}, Ann. Math. Logic \textbf{1} (1970), no.~4,
  389--422.

\bibitem{ghasemi2014reduced}
S.~Ghasemi, \emph{Reduced products of metric structures: a metric
  {F}eferman--{V}aught theorem}, J. Symbolic Logic \textbf{81} (2016), no.~3,
  856--875.

\bibitem{gitik1991strength}
M.~Gitik, \emph{The strength of the failure of the singular cardinal
  hypothesis}, Ann. Pure Appl. Logic \textbf{51} (1991), no.~3, 215--240.

\bibitem{gitikmerimovich}
M.~Gitik and C.~Merimovich, \emph{Power function on stationary classes}, Annals
  of Pure and Applied Logic \textbf{140} (2006), no.~1, 75--103.

\bibitem{halevi2023saturated}
Y.~Halevi and I.~Kaplan, \emph{Saturated models for the working model
  theorist}, Bull. Symb. Logic (2023), 1--5.

\bibitem{harnik1975existence}
V.~Harnik, \emph{On the existence of saturated models of stable theories},
  Proc. Amer. Math. Soc. \textbf{52} (1975), no.~1, 361--367.

\bibitem{Hodg:Model}
W.~Hodges, \emph{Model theory}, Encyclopedia of Mathematics and its
  Applications, vol.~42, Cambridge university press, 1993.

\bibitem{JonssonOlin}
B.~J\'{o}nsson and P.~Olin, \emph{Almost direct products and saturation},
  Compositio Math. \textbf{20} (1968), 125--132 (1968). \MR{227004}

\bibitem{Just:Repercussions}
W.~Just, \emph{Repercussions on a problem of {E}rd\"os and {U}lam about density
  ideals}, Canadian J. Math. \textbf{42} (1990), 902--914.

\bibitem{JustKr}
W.~Just and A.~Krawczyk, \emph{On certain {B}oolean algebras {${\mathcal
  P}(\omega)/I$}}, Trans. Amer. Math. Soc. \textbf{285} (1984), 411--429.

\bibitem{Kana:Book}
A.~Kanamori, \emph{The higher infinite: large cardinals in set theory from
  their beginnings}, Perspectives in Mathematical Logic, Springer,
  Berlin--Heidelberg--New York, 1995.

\bibitem{Keisler.Ultraproducts}
J.~Keisler, \emph{Ultraproducts which are not saturated}, J. Symbolic Logic
  \textbf{32} (1967), 23--46.

\bibitem{malliaris2013general}
M.~Malliaris and S.~Shelah, \emph{General topology meets model theory, on $\fp$
  and $\ft$}, Proc. Natl. Acad. Sci. \textbf{110} (2013), no.~33, 13300--13305.

\bibitem{malliaris2016cofinality}
\bysame, \emph{Cofinality spectrum theorems in model theory, set theory, and
  general topology}, J. Amer. Math. Soc. \textbf{29} (2016), no.~1, 237--297.

\bibitem{malliaris2018keisler}
\bysame, \emph{Keisler's order has infinitely many classes}, Israel J. Math.
  \textbf{224} (2018), no.~1, 189--230.

\bibitem{Mark:Model}
D.~Marker, \emph{Model theory}, Graduate Texts in Mathematics, vol. 217,
  Springer-Verlag, New York, 2002.

\bibitem{merimovich2007power}
C.~Merimovich, \emph{A power function with a fixed finite gap everywhere}, J.
  Symb. Log. \textbf{72} (2007), no.~2, 361--417.

\bibitem{Mija.Boolean}
\v{Z}. Mijajlovi\'{c}, \emph{Saturated {B}oolean algebras with ultrafilters},
  Publ. Inst. Math. (Beograd) (N.S.) \textbf{26(40)} (1979), 175--197.
  \MR{572348}

\bibitem{mitchell2009beginning}
W.J. Mitchell, \emph{Beginning inner model theory}, Handbook of set theory
  (M.~Foreman and A.~Kanamori, eds.), Springer, 2009, pp.~1449--1495.

\bibitem{MoSo}
J.D. Monk and R.M. Solovay, \emph{On the number of complete {B}oolean
  algebras}, Algebra Universalis \textbf{2} (1972), 365--368.

\bibitem{Morley.Cat}
M.~Morley, \emph{Categoricity in power}, Trans. Amer. Math. Soc. \textbf{114}
  (1965), 514--538.

\bibitem{pacholski1970countably}
L.~Pacholski and C.~Ryll-Nardzewski, \emph{On countably compact reduced
  products {I}}, Fund. Math \textbf{67} (1970), no.~1, 155--161.

\bibitem{palmgren}
E.~Palmgren, \emph{A direct proof that certain reduced products are countably
  saturated}, {U.U.D.M.} Report \textbf{27} (1994), 1--7.

\bibitem{palyutin1}
E.~A. Palyutin, \emph{Categorical {Horn} classes {I}}, Algebra Logika
  \textbf{19} (1980), 582--614.

\bibitem{palyutin2}
\bysame, \emph{Categorical {Horn} classes {II}}, Algebra Logika \textbf{49}
  (2010), 782--802.

\bibitem{Pillay.IntroStab}
A.~Pillay, \emph{An introduction to stability theory}, Oxford Logic Guides,
  vol.~8, The Clarendon Press, Oxford University Press, New York, 1983.

\bibitem{schweber}
N.~Schweber, \emph{Does the statement {´There} exists a first-order theory
  {$T$} with no saturated {models´} have any set theoretic strength?},
  MathOverflow, \url{https://mathoverflow.net/q/331463} (version: 2019-05-13).

\bibitem{shelah1969}
S.~Shelah, \emph{Stable theories}, Israel J. Math. \textbf{7} (1969), 187--202.

\bibitem{shelah1972filters}
\bysame, \emph{For what filters is every reduced product saturated?}, Israel J.
  Math. \textbf{12} (1972), 23--31.

\bibitem{Sh:Proper}
\bysame, \emph{Proper forcing}, Lecture Notes in Mathematics 940, Springer,
  1982.

\bibitem{shelah1990classification}
\bysame, \emph{Classification theory and the number of non-isomorphic models},
  Elsevier, 1990.

\bibitem{shelah2009classification}
\bysame, \emph{Classification theory for abstract elementary classes}, College
  Publications, 2009.

\bibitem{shelah2021atomic}
\bysame, \emph{Atomic saturation of reduced powers}, Math. Log. Quarterly
  \textbf{67} (2021), no.~1, 18--42.

\bibitem{ShSte:Non-trivial}
S.~Shelah and J.~Stepr{\=a}ns, \emph{Non-trivial homeomorphisms of
  {$\beta{\mathbb N}\setminus {\mathbb N}$} without the continuum hypothesis},
  Fundamenta Mathematicae \textbf{132} (1989), 135--141.

\bibitem{Sol:Analytic}
S.~Solecki, \emph{Analytic ideals}, Bull. Symb. Logic \textbf{2} (1996),
  339--348.

\bibitem{To:Gaps}
S.~Todor{\v{c}}evi\`c, \emph{Gaps in analytic quotients}, Fund. Math
  \textbf{156} (1998), 85--97.

\bibitem{wierzejewski}
J.~Wierzejewski, \emph{On stability and products}, Fundamenta Mathematicae
  \textbf{93} (1976), 81--95.

\end{thebibliography}
\end{document}